\documentclass[a4paper, pdftex]{amsart}
\AtBeginDocument{%
  \def\MR#1{}
}

\usepackage{tikz, tikz-cd, tikz-3dplot}
\usetikzlibrary{calc}
\usepackage{amsmath,amssymb, euscript, comment}
\usepackage{amsthm}
\usepackage[colorlinks=false]{hyperref}
\usepackage{mathrsfs}
\usepackage[shortlabels]{enumitem}
\usepackage{mathtools}
\usepackage[all,cmtip,2cell]{xy}
\usepackage{array}
\usepackage{comment}
\usepackage{braket}
\usepackage{stmaryrd}
\usepackage{url}
\usepackage{ascmac, multirow}
\usetikzlibrary{calc}

\usepackage{listings}
\usepackage{xcolor}
\usepackage{tcolorbox} 
\definecolor{codegray}{gray}{0.9}
\definecolor{terminalbg}{rgb}{0.95, 0.95, 0.95}
\definecolor{terminaltext}{rgb}{0.1, 0.1, 0.1}
\definecolor{terminalprompt}{rgb}{0.4, 0.4, 0.4}

\definecolor{bgcolor}{RGB}{250,250,250}
\definecolor{codegray}{gray}{0.3}
\definecolor{keywordcolor}{RGB}{86,156,214}
\definecolor{stringcolor}{RGB}{214,157,133}
\definecolor{commentcolor}{RGB}{87,166,74}
\definecolor{identifiercolor}{RGB}{36,36,36}

\lstdefinestyle{modernpython}{
  breaklines=true,
  breakatwhitespace=false,
  postbreak=\mbox{\textcolor{gray}{$\hookrightarrow$}\space},
  columns=flexible,
  backgroundcolor=\color{bgcolor},
  language=Python,
  basicstyle=\ttfamily\small\color{identifiercolor},
  keywordstyle=\color{keywordcolor}\bfseries,
  stringstyle=\color{stringcolor},
  commentstyle=\color{commentcolor},
  showstringspaces=false,
  frame=single,
  framerule=0pt,
  rulecolor=\color{gray!20},
  tabsize=2,
  captionpos=b,
  breaklines=true,
  breakatwhitespace=true,
  numbers=none
}

\lstdefinestyle{terminal}{
  backgroundcolor=\color{terminalbg},
  basicstyle=\ttfamily\small\color{terminaltext},
  frame=single,
  rulecolor=\color{gray},
  frameround=tttt,
  breaklines=true,
  showstringspaces=false,
  xleftmargin=1em,
  aboveskip=1em,
  belowskip=1em,
}

\usepackage{setspace}
\newcommand{\marginparstretch}{0.6}
\let\oldmarginpar\marginpar
\renewcommand\marginpar[1]{\-\oldmarginpar[\framebox{\setstretch{\marginparstretch}\begin{minipage}{\marginparwidth}{\raggedleft\tiny #1}\end{minipage}}]{\framebox{\setstretch{\marginparstretch}\begin{minipage}{\marginparwidth}{\raggedright\tiny #1}\end{minipage}}}}

\usepackage[top=30truemm,bottom=30truemm,left=35truemm,right=35truemm]{geometry}

\theoremstyle{plain}
\newtheorem{thm}{Theorem}[section]
\newtheorem{prop}[thm]{Proposition}
\newtheorem{lem}[thm]{Lemma}

\newtheorem{cor}[thm]{Corollary}
\newtheorem*{thm*}{Theorem}

\theoremstyle{definition}
\newtheorem{defi}[thm]{Definition}
\newtheorem{conj}[thm]{Conjecture}
\newtheorem*{NaC}{Notation and Convention}

\theoremstyle{remark}
\newtheorem{rem}[thm]{Remark}
\newtheorem{nota}[thm]{Notation}

\numberwithin{equation}{section}

\newcommand{\G}{\mathbb{G}}
\newcommand{\Z}{\mathbb{Z}}

\newcommand{\C}{\mathbb{C}}
\newcommand{\A}{\mathbb{A}}
\renewcommand{\P}{\mathbb{P}}

\newcommand{\bS}{\mathbb{S}}

\DeclareMathOperator{\bV}{\mathbf{V}}

\newcommand{\arw}{\longrightarrow}
\newcommand{\diag}{\operatorname{diag}}

\newcommand{\res}{\mathbf{res}}

\DeclareMathOperator{\id}{id}

\DeclareMathOperator{\Sym}{\mathrm{Sym}}

\DeclareMathOperator{\Spec}{\mathrm{Spec}}

\DeclareMathOperator{\Proj}{\mathrm{Proj}}

\DeclareMathOperator{\Gr}{\mathrm{Gr}}
\DeclareMathOperator{\Fl}{\mathrm{Fl}}

\DeclareMathOperator{\Bl}{\mathrm{Bl}}

\DeclareMathOperator{\codim}{\mathrm{codim}}

\DeclareMathOperator{\Hom}{Hom}
\DeclareMathOperator{\End}{End}

\DeclareMathOperator{\Ext}{Ext}

\DeclareMathOperator{\Tot}{Tot}
\DeclareMathOperator{\refl}{refl}
\DeclareMathOperator{\thick}{thick}

\newcommand{\mcD}{\mathcal{D}}
\newcommand{\mcE}{\mathcal{E}}
\newcommand{\mcF}{\mathcal{F}}
\newcommand{\mcG}{\mathcal{G}}

\newcommand{\mcO}{\mathcal{O}}

\newcommand{\mcQ}{\mathcal{Q}}

\newcommand{\mcT}{\mathcal{T}}
\newcommand{\mcU}{\mathcal{U}}

\newcommand{\mcX}{\mathcal{X}}

\newcommand{\scrA}{\EuScript{A}}
\newcommand{\scrB}{\EuScript{B}}
\newcommand{\scrC}{\EuScript{C}}
\newcommand{\scrD}{\EuScript{D}}
\newcommand{\scrE}{\EuScript{E}}
\newcommand{\scrF}{\EuScript{F}}
\newcommand{\scrG}{\EuScript{G}}

\newcommand{\scrO}{\EuScript{O}}
\newcommand{\scrP}{\EuScript{P}}

\newcommand{\scrT}{\EuScript{T}}
\newcommand{\scrU}{\EuScript{U}}
\newcommand{\scrV}{\EuScript{V}}
\newcommand{\scrW}{\EuScript{W}}

\DeclareMathOperator{\Qcoh}{Qcoh}
\DeclareMathOperator{\Perf}{Perf}
\DeclareMathOperator{\coh}{coh}

\DeclareMathOperator{\RG}{\mathrm{R}\Gamma}

\newcommand{\Gm}{\mathbb{G}_{\mathrm{m}}}

\DeclareMathOperator{\RHom}{\mathrm{RHom}}
\DeclareMathOperator{\RlcHom}{R\mathcal{H}\mathit{om}}

\DeclareMathOperator{\D}{D}
\DeclareMathOperator{\Db}{\mathrm{D}^{\mathrm{b}}}
\DeclareMathOperator{\Dcoh}{Dcoh}
\DeclareMathOperator{\Dmod}{Dmod}

\DeclareMathOperator{\GL}{\mathrm{GL}}

\renewcommand{\mod}{\mathop \mathrm{mod}}

\newcommand{\wX}{\widetilde{X}}

\newcommand{\quot}{\mathord{/\!\!/}}

\makeatletter
\def\xof#1{\pgf@process{#1}\pgfmathparse{\pgf@x/65536}\pgfmathresult}
\def\yof#1{\pgf@process{#1}\pgfmathparse{\pgf@y/65536}\pgfmathresult}
\makeatother

\mathchardef\mhyphen="2D

\title[Window categories for a simple $9$-fold flop of Grassmannian type]{Window categories for a simple $9$-fold flop \\ of Grassmannian type}
\author[W. Donovan, W. Hara, M. Kaputska, M. Rampazzo]{Will Donovan, Wahei Hara, Micha\l{} Kapustka, Marco Rampazzo}

\date{\today}

\begin{document}
\maketitle

\begin{abstract}
    The local simple $9$-fold flop of Grassmannian type is a birational transformation between total spaces of vector bundles on the Grassmannians $\Gr(2, 5)$ and~$\Gr(3, 5)$. We produce four different derived equivalences which commute with the pushforward functors for the flopping contractions. These equivalences are realized by identifying four different window categories inside the derived category of coherent sheaves on an Artin stack. As an application, our approach provides a new proof of derived equivalence for a pair of non-birational Calabi--Yau threefolds realized as zero loci of sections of homogeneous vector bundles in Grassmannians.
\end{abstract}

\setcounter{tocdepth}{1}
\tableofcontents

\section{Introduction}

\subsection{Background}

Consider a diagram for a flop between two smooth varieties.
\begin{equation} \label{diagram of flop}
\begin{tikzcd}
X_+ \arrow[dr, "f_+"'] \arrow[rr, dashrightarrow] & & X_- \arrow[dl, "f_-"] \\
& Y&
\end{tikzcd}
\end{equation}
Then by definition of a flop there is no $Y$-isomorphism between $X_+$ and $X_-$,
i.e.~there is no isomorphism $\varphi \colon X_+ \xrightarrow{\sim} X_-$ such that $f_+ = f_- \circ \varphi$.
On the contrary, such an ``isomorphism'' can exist at the level of derived categories.
Namely, it is expected, originally by Bondal and Orlov, that the following holds.

\begin{conj} \label{D-equivalence conj}
Given a smooth flop diagram as in (\ref{diagram of flop}), there exists an exact equivalence 
\[ \Phi \colon \mcD(X_+) \xrightarrow{\sim} \mcD(X_-) \] 
with a functor isomorphism $R(f_+)_* \simeq R(f_-)_* \circ \Phi$,
where $\mcD(X_{\pm}) \coloneqq \Db(\coh X_{\pm})$.
 \end{conj}
 
This conjecture is known to be true in dimension three \cite{bridgeland_flop, VdB_flop},
and indicates an interesting relationship between derived categories and birational geometry.
Note that there exists a generalisation of the conjecture above to the context of K-equivalences, 
which is known as the DK-conjecture \cite{bondalorlov, kawamata}.

Conjecture~\ref{D-equivalence conj} in dimension more than or equal to four is still widely open.
One major obstacle in addressing this problem arises from the inherent difficulty of classifying (smooth) flops.
Within the overwhelmingly diverse and complex landscape of flops, 
one reasonable class to study would be given by \textit{simple flops},
which are smooth flops connected by one smooth blow-up and one smooth blow-down.
The Atiyah flops and Mukai flops are typical examples of simple flops.

\subsection{Simple flops}
Kanemitsu \cite{kanemitsu} provided a large list of simple flops 
and at the same time classified simple flops up to dimension eight.
Kanemitsu characterised the local structure of simple flops using \textit{simple Mukai pairs}.
A Mukai pair is a pair $(V,\mcE)$ of a Fano manifold $V$ and an ample vector bundle $\mcE$ over $V$ such that $c_1(V) = c_1(\mcE)$,
and a Mukai pair $(V, \mcE)$ is called simple if $V$ is of Picard rank one and $\P_V(\mcE) = \Proj_V \Sym^{\bullet} \mcE$ has another $\P^{r-1}$-bundle structure, where $r$ is the rank of $\mcE$.
Given a simple Mukai pair $(V_+, \mcE_+)$, the projectivization $\P_{V_+}(\mcE_+)$ is called a \textit{roof},
and the second $\P^{r-1}$-bundle structure of the roof is given by another Mukai pair $(V_-, \mcE_-)$. 
In a simple flop $X_+ \dashrightarrow X_-$, the ambients $V_{\pm}$ of Mukai pairs appear as the centres of blow-ups, and the roof $\P(\mcE_+) = \P(\mcE_-)$ appears as the common exceptional divisor, as shown below.
\[
\begin{tikzcd}
& & \P_{V_+}(\mcE_+) = \P_{V_-}(\mcE_-) \arrow[ddrr] \arrow[ddll] \arrow[d, hook] & & \\
& & \Bl_{V_+}X_+ = \Bl_{V_-}X_- \arrow[dr] \arrow[dl] & & \\
V_+ \arrow[r, hook] & X_+ \arrow[rr, dashrightarrow] & & X_- & V_- \arrow[l, hook']
\end{tikzcd}
\]
In all known cases, the ambient $V$ of a simple Mukai pair $(V,\mcE)$ is a rational homogeneous manifold, 
and $\mcE$ is an irreducible homogeneous vector bundle\footnote{Remarkably, only one example of a roof $\P_V(\mcE)$ is known which is not given by a rational homogeneous manifold~\cite{kanemitsu}.}.
Following this fact, Kanemitsu labeled simple flops in his classification using Dynkin diagrams.
For instance, Atiyah flops are of type $A_n \times A_n$, and Mukai flops are of type $A_n^M$.

The significance of studying simple flops lies not only in the richness and simplicity of this class within birational geometry, but also in its connection to compact Calabi--Yau manifolds.
Namely, for a simple Mukai pair $(V_+,\mcE_+)$ with the corresponding Mukai pair~$(V_-, \mcE_-)$,
there exists a natural isomorphism $H^0(\mcE_+) \simeq H^0(\mcE_-)$,
and if one chooses two general regular sections $s_{\pm} \in H^0(\mcE_{\pm})$
that are related to each other by this isomorphism, 
then as the vanishing loci $\bV(s_{\pm}) \subset V_{\pm}$
one has a pair $(\bV(s_+), \bV(s_-))$ of compact Calabi--Yau manifolds.
Such Calabi--Yau manifolds $\bV(s_+)$ and $\bV(s_-)$ are not birationally equivalent to each other in general, 
but are expected to satisfy some homological or motivic relations such as D-equivalence or L-equivalence.
The birational transformation $\Tot_{V_+}(\mcE_+^{\vee}) \dashrightarrow  \Tot_{V_-}(\mcE_-^{\vee})$ is a local model of the simple flop of the corresponding type, 
and if there is a reasonable derived equivalence 
$\mcD(\Tot_{V_+}(\mcE_+^{\vee})) \xrightarrow{\sim} \mcD(\Tot_{V_-}(\mcE_-^{\vee}))$, 
then one can descend it to give a derived equivalence 
$\mcD(\bV(s_+)) \xrightarrow{\sim} \mcD(\bV(s_-))$.

\subsection{Simple flop of type $A_4^G$} \label{sect:AG4flop}

The aim of this paper is to study the derived equivalence for a local model of the simple flop of type $A_4^G$
from the perspectives of geometric invariant theory and noncommutative algebra.
The corresponding simple Mukai pairs and roof are related to the following marked Dynkin datum.

\[ \begin{tikzpicture}[
  every node/.style={circle, draw, minimum size=2mm, inner sep=0pt},
  label distance=4pt
]

  \node (1) at (0,0) [label=below:{$\alpha_1$}] {};
  \node (2) at (1.5,0) [fill=black, label=below:{$\alpha_2$}] {};
  \node (3) at (3,0) [fill=black, label=below:{$\alpha_3$}] {};
  \node (4) at (4.5,0) [label=below:{$\alpha_4$}] {};
  
  \draw (1) -- (2) -- (3) -- (4);

\end{tikzpicture} \]

More explicitly, the Mukai pairs are given by $(\Gr(3,5), \mcU_3(2))$ and $(\Gr(2,5), \mcQ_3^{\vee}(2))$,
where $\Gr(k,n)$ denotes the Grassmannian of $k$-dimensional linear subspaces in $\C^n$, $\mcU_3$ is the rank three universal subbundle over $\Gr(3,5)$, and $\mcQ_3$ is the rank three universal quotient bundle over $\Gr(2,5)$.
The roof is
\[ \Fl(2,3;5) \simeq \P_{\Gr(3,5)}(\mcU_3(2)) \simeq \P_{\Gr(2,5)}(\mcQ_3^{\vee}(2)), \]
which is actually the homogeneous manifold corresponding to the above Dynkin diagram with markings.

Consider the total spaces
\begin{align*}
X_+ &\coloneqq \Tot_{\Gr(3,5)}(\mcU_3^{\vee}(-2)) \coloneqq \Spec \Sym^{\bullet} (\mcU_3(2)), \\
X_- &\coloneqq \Tot_{\Gr(2,5)}(\mcQ_3(-2)) \coloneqq \Spec \Sym^{\bullet} (\mcQ_3^{\vee}(2)), 
\end{align*}
and their zero-sections $\G_+ \subset X_+$ and $\G_- \subset X_-$.
By definition, there are canonical isomorphisms
\[ \Bl_{\G_+}X_+ \simeq \Tot_{\Fl(2,3;5)}(\mcO(-1,-1)) \simeq \Bl_{\G_-}X_-, \]
and the associated birational map
\[ X_+ \dashrightarrow X_- \]
is a local model of the simple flop of type $A^G_4$.
Let $f_{\pm} \colon X_{\pm} \to \Spec R$ be the flopping contractions, where $R \simeq H^0(\mcO_{X_+}) \simeq H^0(\mcO_{X_-})$.
The ring $R$ is isomorphic to the coordinate ring of the affine cone of the polarized variety $(\Fl(2,3;5), \mcO(1,1))$.

Now we are ready to state the main result of this paper.

\begin{thm}[\ref{cor:main result}]\label{thm:main_theorem}
There exists an exact equivalence of $R$-linear triangulated categories
\[ \Phi \colon \mcD(X_+) \xrightarrow{\sim} \mcD(X_-) \] 
that satisfies a functor isomorphism $R(f_+)_* \simeq R(f_-)_* \circ \Phi$.
\end{thm}

Note that the equivalence of the derived categories (as $R$-linear triangulated categories) for this flop was first constructed by Morimura \cite{morimura_flop}. 
The method of Morimura follows a sequence of mutations in the derived category of  $\Tot_{\Fl(2,3;5)}(\mcO(-1,-1))$  parallel to a sequence of mutations on the hyperplane section of the variety $\Fl(2,3;5)$ used in \cite{ourpaper_cy3s} to prove the derived equivalence for the associated pair of Calabi--Yau manifolds.

In this paper, we prove the derived equivalence for the flop by employing tilting bundles, which provides a completely different approach from that of Morimura. 
Unlike Morimura's equivalence, our construction allows one to determine immediately the images of certain standard vector bundles, including some line bundles, under the equivalence. 
Moreover, we construct four derived equivalences satisfying the functor isomorphism in Conjecture~\ref{D-equivalence conj}; 
the appearance of exactly four such equivalences is expected to have a physical background (see Remark~\ref{rem:Hori}). 
A further advantage of our approach is that starting from a tilting bundle one can naturally construct important objects such as windows and noncommutative crepant resolutions.

\subsection{Strategy}\label{sec:strategy}
As already mentioned, 
the main idea behind the proof of Theorem \ref{thm:main_theorem} is to use \textit{tilting bundles}
(see Definition~\ref{def:tilting}).
Namely, if one has tilting bundles $\scrT_+$ and $\scrT_-$ over $X_+$ and $X_-$ respectively, 
such that $\End_{X_+}(\scrT_+) \simeq \End_{X_-}(\scrT_-)$,
then they yield derived equivalences
\[ \mcD(X_+) \xrightarrow[{\RHom(\scrT_+,-)}]{\sim} \mcD(\End_{X_{\pm}}(\scrT_{\pm})) \xrightarrow[{- \otimes \scrT_-}]{\sim} \mcD(X_-). \]
It is worth noting that for the simple flop of type $A^G_4$
the existence of tilting bundles has a deeper background from geometric invariant theory (GIT),
which can be summarized as follows.
To begin with, there exists a $\GL(3)$-representation $W$ such that its GIT quotients with respect to two stability conditions
are $W \quot_\pm\GL(3) \simeq X_\pm$.
In particular, there are canonical open immersions 
$\iota_\pm \colon X_\pm \hookrightarrow [W/\GL(3)]$
where $[W/\GL(3)]$ is the corresponding Artin stack.
Furthermore, for the tilting bundles $\scrT_{\pm}$ constructed in this paper,
there is a single pretilting bundle $\scrT$ over $[W/G]$ whose restrictions recover them,
i.e.~$\iota_\pm^* \scrT \simeq \scrT_\pm$.
This bundle $\scrT$ generates a thick subcategory $\scrW \subset \mcD([W/G])$
such that the restriction functors $\iota_{\pm}^* \colon \scrW \to \mcD(X_{\pm})$ are both equivalences.

Such a subcategory $\scrW \subset \mcD([W/G])$ is often referred to as a \textit{window category} (for both $X_+$ and $X_-$).
Several generalities for window categories have been established in previous works including \cite{BFK_VGIT,halpernleistner_GIT, halpernleistner_derived}.
However, note that the results in \cite{halpernleistner_derived} cannot be applied to our example of a flop,
since the $\GL(3)$-representation $W$ is not quasi-symmetric.
Moreover, it turns out that the window categories introduced in \cite{BFK_VGIT,halpernleistner_GIT} do not coincide with the window categories constructed in this paper, as discussed in \S\ref{sect:comparisonHL}.

The above background involving Artin stacks and window categories can be applied
to examples that admit a description as a GIT quotient of a smooth affine variety (or vector space).
Among the known simple flops, such GIT quotient descriptions are available only for the Atiyah flops (i.e., of type $A_n \times A_n$)
and the flops of type $A^G_n$.
The approach using tilting bundles is broadly parallel to the study of window categories in cases where a GIT description exists,
but importantly, it also extends to more general situations.
In fact, a proof of derived equivalence via tilting bundles is known for the simple flops listed in the table below. At present, such derived equivalences have been established for all known simple flops up to dimension ten.

\begin{table}[h]
    \label{tab:flops}
    \renewcommand{\arraystretch}{1.3} 
    \centering
    \begin{tabular}{|c|c|c|}
        \hline
        Dimension & Type of simple flop & Equivalence using tilting bundles \\ \hline\hline
        $2n+1$ ($n \geq 1$) & $A_n \times A_n$ (Atiyah flop) & well-known (c.f.~\cite{hara_mukai}) \\ \hline
        $2n$ ($n \geq 2$) & $A_n^M$ (Mukai flop) & well-known (c.f.~\cite{hara_mukai}) \\ \hline
                     5              & $C_2$ (Abuaf flop)  & \cite{segal_C2, hara_abuaf} \\ \hline
                     7              & $G_2$ (Abuaf--Ueda flop) & \cite{hara_G2} \\ \hline
                     8              & $G_2^{\dagger}$ & \cite{hara_G2dag} \\ \hline
                     9              & $A_4^G$ & this paper \\ \hline
                     10		    & $D_4$ & \cite{hara_D4} \\ \hline
        \vdots & \vdots & \vdots \\ \hline
    \end{tabular}
    
\bigskip
        \caption{Derived equivalence of known simple flops using tilting bundles.}
\end{table}

\subsection{Projective Calabi--Yau threefolds and GLSM}
In theoretical physics, 
the data of the variation of GIT relating $X_+$ and $X_-$ describes a \emph{gauged linear sigma model} (GLSM). 
Such entities are supersymmetric gauge theories introduced by Witten in \cite{witten}, 
which exhibit different phases: roughly speaking, each of those phases is associated to a space parametrizing lowest-energy states, which is the critical locus of a function called \emph{superpotential} (determined by the geometric data).
By physical arguments, it is expected that such critical loci are related at the level of derived categories of coherent sheaves. 
While Witten's original formulation describes GIT quotients with respect to abelian groups, successive advancements have been made in the context of non-abelian gauge theories, 
producing GLSM interpretations of Calabi--Yau subvarieties of Grassmannians and Pfaffians, 
and other interesting geometries \cite{horitong, addingtondonovansegal, knappsharpe}.
In particular, our variation of GIT has been shown to exhibit two ``simple'' geometric phases \cite{ourpaper_cy3s}, 
whose associated critical loci are Calabi--Yau threefolds which are non-birationally equivalent. 
There is a well-known correspondence between zero loci of sections of vector bundles and associated GLSM phases: 
in particular, a general section $s$ of a globally generated vector bundle $\mcE$ on $X$ determines a GLSM phase with critical locus $\bV(s)$ and a canonical superpotential $Q_s:\operatorname{Tot}(\mcE^\vee)\arw \C$. 
In such context, by Kn\"orrer periodicity, 
the derived category of $\bV(s)$ is equivalent to the derived factorization category $\Dcoh_{\Gm}(\operatorname{Tot}(\mcE^\vee), \chi_{\id}, W_s)$ (see \S\ref{sect:CY3results}).
This correspondence enables us, via Theorem~\ref{thm:main_theorem}, to provide a new proof of the main result of \cite{ourpaper_cy3s} and to extend the result to the singular case of $\bV(s)$, by arguments similar to those in \cite{segal_git, addingtondonovansegal}.

\begin{cor}[\ref{cor:CY3Deq}]
Let
\begin{align*}
s_+ & \in H^0(\Gr(3,5), \mcU_3(2)) \\
s_- & \in H^0(\Gr(2,5), \mcQ_3^{\vee}(2))
\end{align*}
be corresponding sections under the canonical isomorphism $H^0(\mcU_3(2)) \simeq H^0(\mcQ_3^{\vee}(2))$.
If $s_{\pm}$ are regular,
then there is an exact equivalence of the derived categories
\[ \mcD(\bV(s_+)) \simeq \mcD(\bV(s_-)) \]
of (possibly singular) Calabi--Yau threefolds.
\end{cor}

\begin{rem} \label{rem:Hori}
The gauge theory in physics can describe concretely
 the stringy K\"{a}hler moduli space (SKMS) associated to the GLSM we study in this paper.
 We provide a rough outline here; for more detailed account of the theory and the related works, 
 we refer the reader to \cite{HHP,Hori, EHKR17,EHKRnew}.

First, the model involves two parameters, the Fayet--Iliopoulos (FI) parameter and the theta parameter. 
Since the theta parameter is $2\pi$-periodic, it is natural to consider the FI--theta parameter space modulo this periodicity, 
which can be regarded as a cylinder or, equivalently, a twice-punctured sphere. 
To obtain the SKMS, one further removes three points on the cylinder corresponding to loci where the model is ill-defined, up to $2\pi$ shifts. 
In particular, the resulting SKMS is a five-punctured sphere.

The regions where the absolute value of the FI parameter is large correspond to the two phases of the model, namely the Calabi--Yau threefolds $\bV(s_+)$ and $\bV(s_-)$. 
Homotopy classes of paths connecting points in these regions (denoted $m_+$ and $m_-$ in Figure~\ref{fig:SKSM}) give rise to the physical notion of windows. 
For the GLSM considered in this paper, one can identify four fundamental such windows (corresponding to paths A, B, C and D in Figure~\ref{fig:SKSM}).

In this paper, we construct four window \emph{categories} to establish derived equivalences for the flop of type $A_4^G$, which we expect to correspond to the four fundamental windows in the physical sense. 
We hope that further investigation of this example from the perspective of physics will be pursued, 
and that the mathematical results in this paper will contribute to further deepening the relation between physics and mathematics.
\end{rem}

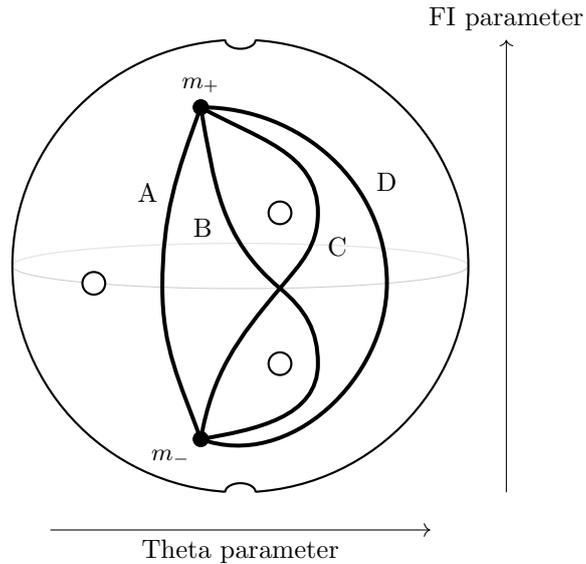
\begin{figure} 
\begin{tikzpicture}
  \def\c{3}
  \def\r{0.2}
  \def\p{0.1}
  \def\m{6}

  \draw[black!10, line width=0.5pt]
    plot[domain=0:180, smooth, variable=\t]
      ({\c*cos(\t)}, {\c*\p*sin(\t)});
  \draw[black!20, line width=0.3pt]
    plot[domain=180:360, smooth, variable=\t]
      ({\c*cos(\t)}, {\c*\p*sin(\t)});

  \coordinate (P) at ({3*cos(230)}, {0.3*sin(230)});
  \draw[fill=white, thick] (P) circle [radius=0.15];
  
  \draw[thick] (0,0) circle (\c);
  \fill[white] (0,\c) circle (\r); 
  \fill[white] (0,-\c) circle (\r);
  \draw[thick]
    plot[domain=180:360, smooth, variable=\t]
      ({\r*cos(\t)}, {\c+\m*\r*\p*sin(\t)});  
  \draw[thick]
    plot[domain=180:360, smooth, variable=\t]
      ({\r*cos(\t)}, {-\c-\m*\r*\p*sin(\t)});  

  \coordinate (Q) at ({3*cos(280)}, {0.3*sin(280)});
  \coordinate (Qup) at ($(Q)+(0,1)$);
  \coordinate (Qdown) at ($(Q)+(0,-1)$);
  \draw[fill=white, thick] (Qup) circle [radius=0.15];
  \draw[fill=white, thick] (Qdown) circle [radius=0.15];
  
  \coordinate (M) at ({3*cos(260)}, {0.3*sin(260)});
  \coordinate (M+) at ($(M)+(0,2.4)$);
  \coordinate (M-) at ($(M)+(0,-2)$);
  \draw[fill=black] (M+) circle [radius=0.1] node[above=2pt] {\small $m_+$};
  \draw[fill=black] (M-) circle [radius=0.1] node[below left=0.5pt] {\small $m_-$};
    
  \coordinate (b) at ({3*cos(250)}, {0.3*sin(250)});
  \coordinate (a) at ($(Qdown)+(0.5,0)$);
  \coordinate (c) at ($(Qup)+(0.5,0)$);
  \coordinate (d) at ({3*cos(310)}, {0.3*sin(310)});
  
  \draw[line width=1.5pt] 
  (M+) to[out=250,in=90,relative=false] node[auto=right, pos=0.6]{A} (b) to[out=270,in=110,relative=false] (M-)
  ;
  \draw[line width=1.5pt] 
  (M+) to[out=280,in=140,relative=false] node[auto=right, pos=0.5]{B} (Q) to[out=320,in=90,relative=false] (a) to[out=270,in=10,relative=false] (M-)
  ;
  \draw[line width=1.5pt] 
  (M+) to[out=330,in=90,relative=false]  (c) to[out=270,in=50,relative=false] node[auto=left, pos=0.15]{C} (Q) to[out=230,in=80,relative=false] (M-)
  ;
  \draw[line width=1.5pt] 
  (M+) to[out=0,in=90,relative=false] node[auto=left, pos=0.7]{D} (d) to[out=270,in=340,relative=false] (M-)
  ;
  
  \draw[->] (3.5,-3) to (3.5,3) node[above]{FI parameter};
  \draw[->] (-2.5,-3.5) to node[auto=right]{Theta parameter} (2.5,-3.5) ;
\end{tikzpicture} 
\caption{Sketch of the SKMS with paths corresponding to windows.}
\label{fig:SKSM}
\end{figure}

\begin{rem} \label{rem:PFeq}
Another aspect --  related to Remark~\ref{rem:Hori} -- of the appearance of four windows is the perspective offered by mirror symmetry between compact Calabi--Yau manifolds. Let us explain this point of view in our example.

In our case, the compact Calabi--Yau threefolds associated with the flop belong to the family $\mathcal X_{25}$ of complete intersections of two Grassmannian translates (c.f.~\cite[Section~5]{ottemrennemo}), 
whose mirror family is conjecturally governed by Picard--Fuchs equation no.~101 in the database \cite{vEvS} (see \cite[Section 4]{KapPfaffian}). 
This equation has two points of maximal unipotent monodromy (MUM) exchanged by an automorphism of the equation, 
which is related to the existence of derived equivalences between pairs of elements of the family. 
In our case, these are pairs $(\mcD(\bV(s_+)),\mcD(\bV(s_-)))$. 
Moreover, the equation has in total six singular points, namely the two MUM points, three conifold singularities and one apparent singularity. Since the latter does not induce any monodromy,
we can assume that the complex moduli space is supported on a five-punctured sphere. 
By mirror symmetry, 
the complex moduli space of $\mathcal{Y}$ should correspond to the SKMS of elements of $\mathcal{X}_{25}$. 
By homological mirror symmetry, the monodromy operators of the Picard--Fuchs equation should correspond to autoequivalences of the derived category of elements of $\mathcal{X}_{25}$. 
On the other hand, as written in Remark~\ref{rem:Hori}, the four windows provide four equivalences between $\mcD(\bV(s_+))$ and $\mcD(\bV(s_-))$, 
which correspond to four paths in the SKMS connecting $\bV(s_+)$ and $\bV(s_-)$.

By composing any such equivalence with the inverse of another, we obtain autoequivalences of $\mcD(\bV(s_+))$ corresponding to loops in the SKMS. 
These loops correspond to loops around singular points of the complex moduli space of the mirror which in turn induce monodromy operators around the three non-MUM degeneracy points of the Picard--Fuchs equation. 
However, the precise rule associating autoequivalences of $\mcD(\bV(s_+))$ with monodromy operators on the mirror side remains to be determined.
\end{rem}

\subsection*{Structure of the paper}
\S\ref{sect:tilting prelim} reviews the notions of NCCRs and tilting objects, which constitute fundamental tools in the construction of derived equivalences, and provides several auxiliary lemmas that will be useful in finding tilting objects.
\S\ref{sec:excceptional_collection} is devoted to the exceptional collections on the Grassmannian that serve as the starting point for our construction of tilting bundles. 
The collections required for our purposes are not among the standard ones: 
they are obtained by applying mutations to Kapranov's collection.
\S\ref{sect:geom of flop} recalls the construction of flops via GIT and summarizes their geometry using diagrams.
\S\ref{sec:windows} develops the central part of the paper, namely the construction of tilting bundles.
\S\ref{sect:conclusion} presents the main results, explaining the consequences of the existence of such bundles for window categories and derived equivalences.
\S\ref{sect:comparisonHL} provides a comparison of our approach with general window category theory, in which experts in this field may be interested.
Finally, Appendix~\ref{sect:fffunctors} gives some general facts concerning fully faithful functors, 
and Appendix \ref{app:code} collects the scripts supporting the computations in \S\ref{sec:windows}.

\begin{NaC}
We work over $\C$, and adopt the following notations.
\begin{enumerate}
\item[$\bullet$] $\mcD(X) \coloneqq \Db(\coh X)$.
\item[$\bullet$] $\mcD(A) \coloneqq \Db(\mod A)$, where $\mod A$ denotes the category of finitely generated right modules over a Noetherian algebra $A$. 
\item[$\bullet$] $\refl(X)$ : the category of reflexive sheaves over a normal variety $X$.
\item[$\bullet$] $\refl(R)$ : the category of reflexive modules over a normal domain $R$.
\item[$\bullet$] $\thick(x) \subset \scrD$ : the smallest thick subcategory of $\scrD$ containing $x \in \scrD$.
\item[$\bullet$] $\Gr(r,N)$ : the Grassmannian of $r$-dimensional subspaces of $\C^N$.
\item[$\bullet$] $\mcU_r$ : the rank $r$ universal subbundle over $\Gr(r,N)$.
\item[$\bullet$] $\mcQ_{N - r}$ : the rank $N - r$ universal quotient bundle over $\Gr(r,N)$.
\item[$\bullet$] $V(\chi)$ : the irreducible representation of $\GL(3)$ with the dominant weight $\chi$.
\item[$\bullet$] $\P(\mcE) \coloneqq \Proj \Sym^{\bullet} \mcE$ : the projectivization of $\mcE$.
\item[$\bullet$] $\Tot(\mcE) = \mathbb{V}(\mcE^{\vee}) \coloneqq \Spec \Sym^{\bullet} \mcE^{\vee}$ : the total space of $\mcE$.
\end{enumerate}
\end{NaC}

\subsection*{Acknowledgements}
W.D., W.H., and M.R.~are grateful to Conan Naichung Leung and IMS at CUHK for organising the meeting that began this collaboration. 
The authors would like to thank Kentaro Hori for sharing many valuable insights from physics, 
which in particular greatly contributed to Remarks~\ref{rem:Hori} and~\ref{rem:PFeq}.
W.H.~would also like to thank Hayato Morimura, and Yukinobu Toda for very helpful conversations. 
M.K.~would like to thank Tymoteusz Chmiel for insight about the Picard-Fuchs equation on the mirror side.  
M.R.~would like to express his gratitude to Ed Segal and Ying Xie for useful insights.

W.D.~was supported by Yau MSC, Tsinghua~U., BIMSA, and China TTP.
W.H.~was supported by World Premier International Research Center Initiative (WPI), MEXT, Japan, 
and by JSPS KAKENHI Grant Number JP24K22829. 
M.K.~was supported by the Polish National Sciences Center project number 2024/55/B/ST1/02409. 
M.R.~was supported by Fonds voor Wetenschappelijk Onderzoek -- Vlaanderen (FWO, Research Foundation -- Flanders) -- G0D9323N, and by the European Research Council (ERC) grant agreement No.~817762.

\section{NCCRs and tilting objects} \label{sect:tilting prelim}

This section gives the definition and fundamental properties of noncommutative crepant resolutions (NCCRs) and tilting objects, which are the main tools in this paper.
Let us start by recalling the definition of noncommutative crepant resolution, 
which serves as a noncommutative analog of crepant resolution.

All rings and schemes in this section are assumed to be  Noetherian, and all schemes separated.

\begin{defi}[{\cite{VdB_NCCR}}]
Let $R$ be a normal Gorenstein domain, and $M$ a reflexive \mbox{$R$-module}.
We say that $M$ gives a \textit{noncommutative crepant resolution (NCCR)} $A = \End_R(M)$ of $R$,
if $A$ has finite global dimension and $A$ is Cohen-Macaulay as an $R$-module.
\end{defi}

Usually the existence and a construction of an NCCR for a given example of $R$ is not obvious.
If $R$ admits a crepant resolution in the usual sense with a tilting complex, as defined below, 
then the existence of an NCCR follows.

\begin{defi} \label{def:tilting}
Let $Z$ be a scheme and $\scrT \in \Perf(Z)$ a perfect complex.
\begin{enumerate}
\item[(1)] $\scrT$ is said to be \textit{pretilting} if $\Ext^i_Z(\scrT,\scrT) = 0$ for all $i \neq 0$.
\item[(2)] $\scrT$ is said to be \textit{tilting} if $\scrT$ is pretilting and generates the unbounded derived category $\D(\Qcoh Z)$, i.e.~if $F \in \D(\Qcoh Z)$ satisfies $\RHom_Z(\scrT, F) = 0$, then $F = 0$.
\end{enumerate}
If $\scrT$ is (quasi-isomorphic to) a locally free sheaf of finite rank, $\scrT$ is called a \textit{(pre)tilting bundle}.
\end{defi}

\begin{rem} \label{remark generation}
It is known by \cite[Theorem~2.1.2 and Theorem~3.1.1]{bondalVdB_generator} that an object $\scrT \in \Perf(Z)$ is a generator of $\D(\Qcoh Z)$ if and only if $\scrT$ is a classical generator of $\Perf(Z)$.
Assume in addition that $Z$ is projective over an affine scheme,
and let $\scrT$ be a pretilting object over $Z$.
Then the condition of $\scrT$ to be a generator of $\D(\Qcoh Z)$ is also equivalent to the condition that $\scrT$ generates $\D^-(\coh Z)$.
See \cite[Lemma 2.2]{hara_def}.
\end{rem}

\begin{prop} \label{tilting prop}
Let $R$ be a normal Gorenstein domain and $f \colon X \to \Spec R$ a projective morphism.
Assume that there exists a tilting bundle $\scrT$ over $X$.
\begin{enumerate}
\item[\rm (1)] Then the functor
\[ \RHom_X(\scrT, -) \colon \mcD(X) \to \mcD( \End_X(\scrT)) \]
gives an exact equivalence of $R$-linear triangulated categories.
\end{enumerate}
Assume in addition that $f \colon X \to \Spec R$ is a small resolution.
Then the following hold.
\begin{enumerate}
\item[\rm (2)] The module $f_*\scrT$ is reflexive and the natural homomorphisms give isomorphisms
\[ \End_X(\scrT) \simeq \End_R(f_*\scrT) \]
of $R$-algebras.
\item[\rm (3)] The reflexive $R$-module $f_*\scrT$ gives an NCCR of $R$.
\end{enumerate}
\end{prop}

\begin{proof}
(1) follows from \cite[Lemma 3.3]{todauehara}.
Since $f$ is small, the (underived) push-forward $f_*$ gives an equivalence of categories $f_* \colon \refl(X) \xrightarrow{\sim} \refl(R)$
of reflexive modules.
Thus (2) follows.
The remaining statement (3) follows from (2) by a standard argument (c.f.~\cite[Theorem~1.5]{iyamawemyssMMA}).
\end{proof}

\begin{prop} \label{generation criteria}
Let $f \colon X \to \Spec R$ be a projective crepant resolution,
and $\scrT$ a pretilting object.
Then $\scrT$ is tilting (i.e.~a generator) if and only if $\End_X(\scrT)$ has finite global dimension.
\end{prop}

\begin{proof}
The ``only if'' part follows from Proposition~\ref{tilting prop}.
Assume that $A = \End_X(\scrT)$ has finite global dimension.
Since $f$ is a proper morphism, one can consider the following two functors.
\begin{align*}
\Phi \coloneqq \RHom(\scrT, -) &\colon \D^-(\coh X) \to \D^-(\mod A) \\
\Psi \coloneqq - \otimes_A^{L} \scrT &\colon \D^-(\mod A) \to \D^-(\coh X)
\end{align*}
The pretilting property of $\scrT$ yields $\Phi \circ \Psi \simeq \id_{\D^-(\mod A)}$ (c.f.~\cite[Lemma 3.3]{todauehara}).
Note that $\Phi$ restricts to give the functor $\mcD(X) \to \mcD(A)$.
In addition, since $A$ has finite global dimension, the functor $\Psi$ sends $\mcD(A)$ to $\mcD(X)$.
Therefore, $\mcD(A)$ is equivalent to an admissible subcategory of $\mcD(X)$.
Now the crepancy of $f$ implies that the category $\mcD(X)$ has no non-trivial semiorthogonal decomposition (c.f.~\cite[Lemma~A.5]{SpenkoVdB_toric2}).
Thus, the functor 
\[ \Psi|_{\mcD(A)} \colon \mcD(A) \to \mcD(X) \]
is essentially surjective, and hence an equivalence.
Note that $\Psi(A) \simeq \scrT$.
Since $A$ is a classical generator of $\mcD(A) = \Perf(A)$, the pretilting object $\scrT$ is a classical generator of $\mcD(X) = \Perf(X)$.
\end{proof}

\begin{lem} \label{lem:dual tilting}
Let $Z$ be a scheme.
If $\scrT$ is a tilting object over $Z$, then so is the derived dual $\scrT^{\vee} \coloneqq \RlcHom_Z(\scrT, \mcO_Z)$.
\end{lem}

\begin{proof}
It is clear that $\scrT^{\vee}$ is contained in $\Perf(Z)$ and is a pretilting object.
Thus it is enough to show that $\scrT^{\vee}$ is a generator.
Since the functor $\RlcHom_Z(-,\mcO_Z)$ is a contravariant autoequivalence of the triangulated category $\Perf(Z)$,
$\scrT^{\vee}$ is a classical generator of $\Perf(Z)$ and hence a generator of $\D(\Qcoh Z)$ by Remark~\ref{remark generation}.
\end{proof}

\section{Exceptional collections}\label{sec:excceptional_collection}

In order to prove that $\mcD(X_+)$ and $\mcD(X_-)$ are equivalent, we will construct suitable subcategories of $\mcD(\mcX)$, for $\mcX$ an Artin stack defined in~\S\ref{sec:git}, and we will prove that they are equivalent to both $\mcD(X_+)$ and $\mcD(X_-)$ using tilting bundles.  
This section explains the exceptional collections required in the construction.

\subsection{Kapranov collection}\label{sect:KapFEC}
First, recall that for $\Gr(3,5)$ we may obtain an exceptional collection \cite{kapranov_grassmannians} by applying the following Schur powers to $\mcU_3^\vee$. We use conventions such that $\mathbb{S}^{1,1,1} \mcU_3^{\vee} \simeq  \wedge^3 \mcU_3^{\vee} = \mcO(1)$, and  $\mathbb{S}^{k,0,0} \mcU_3^{\vee} \simeq  \Sym^k \mcU_3^{\vee}$.

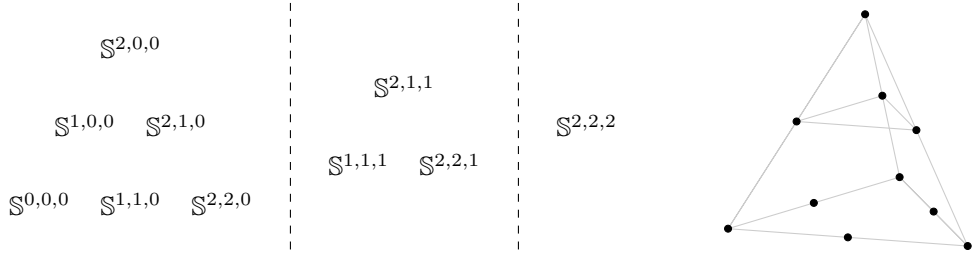
\begin{center}
\begin{tikzpicture}[scale=1.2]

\node at (0,0)   {$\mathbb{S}^{0,0,0}$};
\node at (1,0)   {$\mathbb{S}^{1,1,0}$};
\node at (2,0)   {$\mathbb{S}^{2,2,0}$};
\node at (0.5,0.87) {$\mathbb{S}^{1,0,0}$};
\node at (1.5,0.87) {$\mathbb{S}^{2,1,0}$};
\node at (1,1.73)   {$\mathbb{S}^{2,0,0}$};

\draw[dashed] (2.75,-0.5) -- (2.75,2.25);

\node at (3.5,0.44)   {$\mathbb{S}^{1,1,1}$};
\node at (4.5,0.44)   {$\mathbb{S}^{2,2,1}$};
\node at (4,1.3)     {$\mathbb{S}^{2,1,1}$};

\draw[dashed] (5.25,-0.5) -- (5.25,2.25);

\node at (6,0.87) {$\mathbb{S}^{2,2,2}$};

\end{tikzpicture}
\qquad
\def\dotsize{2pt}
\tdplotsetmaincoords{73}{44} 
\begin{tikzpicture}[scale=0.65,line join=round,tdplot_main_coords,declare function={a=5;}] 
\begin{scope}[canvas is xy plane at z=0,transform shape]
 \path foreach \X [count=\Y] in {A,B,C}
 {(\Y*120:{a/(2*cos(30))}) coordinate(\X)};
\end{scope}
\path (0,0,{a*cos(30)}) coordinate (D);
\draw[gray!40] foreach \X/\Y [remember=\X as \Z (initially D)] in {A/B,B/C,C/D,D/A}
 {(\X) -- (\Z) -- (\Y)};
 
\draw[gray!40] ($ (A)!0.5!(D) $) -- ($ (B)!0.5!(D) $) -- ($ (C)!0.5!(D) $) -- cycle;

\draw[fill=black] (A) circle (\dotsize);
\draw[fill=black] (B) circle (\dotsize);
\draw[fill=black] (C) circle (\dotsize);
\draw[fill=black] (D) circle (\dotsize);

\draw[fill=black] ($ (A)!0.5!(B) $) circle (\dotsize);
\draw[fill=black] ($ (A)!0.5!(C) $) circle (\dotsize);
\draw[fill=black] ($ (A)!0.5!(D) $) circle (\dotsize);
\draw[fill=black] ($ (B)!0.5!(C) $) circle (\dotsize);
\draw[fill=black] ($ (B)!0.5!(D) $) circle (\dotsize);
\draw[fill=black] ($ (C)!0.5!(D) $) circle (\dotsize);

\end{tikzpicture}
\end{center}

The three triangles of Schur powers above may be pictured as successive layers of a tetrahedron, as shown on the right. The Schur powers come from those Young diagrams that fit in a box of height $3$ and width $5-3=2$, as follows.


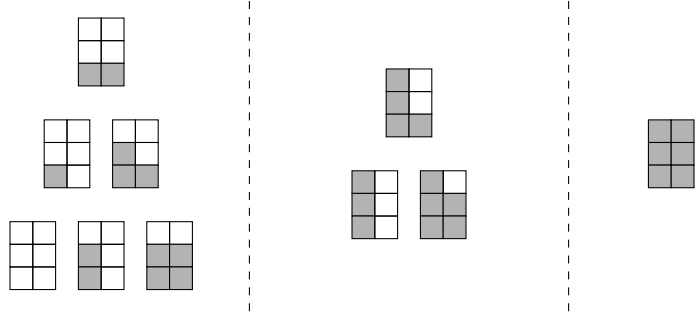
\def\spacex{3}
\def\scaley{0.75}
\begin{center}
\begin{tikzpicture}[scale=0.3]
\definecolor{mygray}{gray}{0.7}

\newcommand{\block}[3]{
  \begin{scope}[shift={(#1,#2)}]
    \foreach \i in {0,1}{
      \foreach \j in {0,1,2}{
        \pgfmathtruncatemacro{\n}{2*\j+\i}
        \foreach \k in {#3} {
          \ifnum\n=\k
            \fill[mygray] (\i,-\j) rectangle ++(1,-1);
          \fi
        }
        \draw (\i,-\j) rectangle ++(1,-1);
      }
    }
  \end{scope}
}

\block{0}{0}{}
\block{3}{0}{2,4}
\block{6}{0}{2,3,4,5}
\block{1.5}{6*\scaley}{4}
\block{4.5}{6*\scaley}{2,4,5}
\block{3}{12*\scaley}{4,5}

\draw[dashed] (9+0.5*\spacex,13*\scaley) -- (9+0.5*\spacex,-4);

\block{13.5+\spacex}{9*\scaley}{0,2,4,5}
\block{12+\spacex}{3*\scaley}{0,2,4}
\block{15+\spacex}{3*\scaley}{0,2,3,4,5}

\draw[dashed] (20+1.5*\spacex,13*\scaley) -- (20+1.5*\spacex,-4);

\block{22+2*\spacex}{6*\scaley}{0,1,2,3,4,5}

\end{tikzpicture}
\end{center}

\subsection{Our collections} \label{sect:ourFEC}

It turns out that Kapranov's collection is the wrong choice for our purpose
(see Remark~\ref{rem:why not Kapranov}). 
To address this, we produce new collections by a sequence of mutations.

\begin{prop} \label{prop our FEC}
The derived category $\mcD(\Gr(3,5))$ admits the following full strong exceptional collections, for $\mcU_3$ the rank three universal subbundle.
\begin{enumerate}
\item[\rm (1)] $\langle \mcO(-1), \mcU_3^{\vee}(-1), \mcO, \mcU_3^{\vee}, \mcU_3(1), \mcO(1), \mcU_3^{\vee}(1), \mcU_3(2), \mcO(2), \mcO(3) \rangle$.
\item[\rm (2)]  $\langle \mcO(-1), \mcU_3^{\vee}(-1), \Sym^2 \mcU_3^{\vee}(-1), \mcO, \mcU_3^{\vee}, \mcU_3(1), \mcO(1), \mcU_3^{\vee}(1), \mcU_3(2), \mcO(2) \rangle$.
\end{enumerate}
\end{prop}

In order to prove this proposition, let us prepare the following lemma.

\begin{lem}\label{lem.seqs}
There are exact sequences of vector bundles over $\Gr(3,5)$ as follows, writing $V\simeq\C^5$ and $\mcU=\mcU_3$ for brevity.
\begin{enumerate}
\item[\rm (1)] $0 \longrightarrow \mcO (-1) \longrightarrow \mcO  \otimes \wedge^3 V \longrightarrow \mcU^\vee \otimes \wedge^4 V \longrightarrow \mathbb{S}^2 \mcU^{\vee} \longrightarrow 0$.

\item[\rm (2)] $0 \longrightarrow \mcU^\vee (-1) \longrightarrow  \mcO \otimes \wedge^2 V \longrightarrow \wedge^2 \mcU^\vee \otimes \wedge^4 V \longrightarrow \mathbb{S}^{2,1} \mcU^\vee \longrightarrow 0$.

\item[\rm (3)] $0 \longrightarrow \mathbb{S}^2 \mcU^\vee  (-1) \longrightarrow \mcU^\vee \otimes \wedge^2 V \longrightarrow \wedge^2 \mcU^\vee \otimes \wedge^3 V \longrightarrow \mathbb{S}^{2,2} \mcU^\vee \longrightarrow 0$.
\end{enumerate}
\end{lem}

\begin{proof} 

The sequences below are exact on $\Gr(3,5)$, where all morphisms are induced by the tautological morphism $\mcU=\mcU_3 \to  \mcO\otimes{V}$ of  bundles on $\Gr(3,V)$. This follows from the twisted Lascoux resolution~\cite{weyman},
see for instance \cite[Theorem~A.7]{Donovan_Segal_2014} for the Grassmannian of quotients $\Gr(V^\vee,3)\cong\Gr(3,V)$, noting that the universal quotient bundle for $\Gr(V^\vee,3)$ corresponds to $\mcU_3^\vee$ under the latter isomorphism.
\begin{enumerate}
\item[\rm (1)] $0 \longrightarrow \mathbb{S}^2 \mcU (-1) \longrightarrow \mcU (-1) \otimes \wedge^4 V^\vee \longrightarrow \mcO (-1) \otimes \wedge^3 V^\vee \longrightarrow \mcO \longrightarrow 0$.
\item[\rm (2)] $0 \longrightarrow \mathbb{S}^{2,1} \mcU (-1) \longrightarrow \wedge^2 \mcU (-1) \otimes \wedge^4 V^\vee \longrightarrow \mcO (-1) \otimes \wedge^2 V^\vee \longrightarrow \mcU \longrightarrow 0$.
\item[\rm (3)] $0 \longrightarrow \mathbb{S}^{2,2} \mcU (-1) \longrightarrow \wedge^2 \mcU (-1) \otimes \wedge^3 V^\vee \longrightarrow \mcU (-1) \otimes \wedge^2 V^\vee \longrightarrow \mathbb{S}^2 \mcU \longrightarrow 0$.
\end{enumerate}
Dualizing and twisting by $\mcO(-1)$ then gives the required sequences.
\end{proof}

\begin{proof}[Proof of Propositon~\ref{prop our FEC}]

To obtain~$(2)$, we mutate three objects of Kapranov's collection, corresponding to one edge of the above tetrahedron, to give three new objects as shown below. For readability, we omit trailing zeroes in the Schur powers. The required mutation sequences are obtained by splitting each of the sequences from Lemma~\ref{lem.seqs} into two short exact sequences.


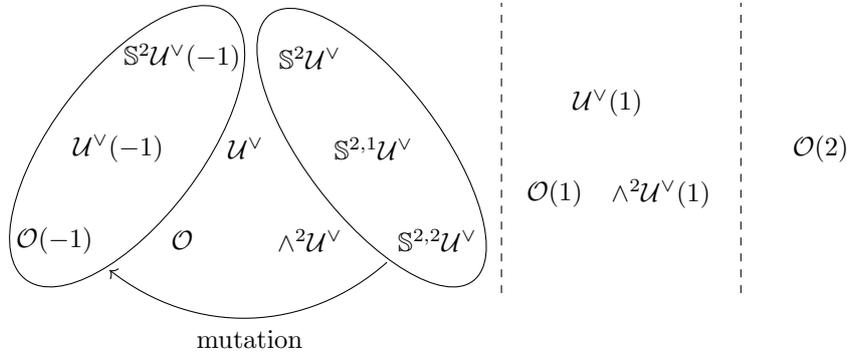
\begin{center}
\begin{tikzpicture}[scale=1.4]

\draw (-1.2,0) node[inner sep=+6pt] (a)  {$\mcO(-1)$};
\draw (0,0) node (b)  {$\mcO$};
\draw (1.2,0) node (c)  {$\wedge^2\mcU^{\vee}$};
\draw (2.4,0) node[inner sep=+4pt] (d)  {$\mathbb{S}^{2,2}\mcU^{\vee}$};
\draw (-0.6,0.87) node (e) {$\mcU^{\vee}(-1)$};
\draw (0.6,0.87) node (f) {$\mcU^{\vee}$};
\draw (1.8,0.87) node (g) {$\mathbb{S}^{2,1}\mcU^{\vee}$};
\draw (1.2,1.73)  node (h) {$\mathbb{S}^2\mcU^{\vee}$};
\draw (0,1.73)  node (i) {$\mathbb{S}^2\mcU^{\vee}(-1)$};

\draw[rotate=53] (0.4,0.95) ellipse (45pt and 20pt);
\draw[rotate=-53] (0.37,1.95) ellipse (45pt and 18pt);

\draw[->] (node cs:name=d, anchor=south west) to[bend left=40] node[yshift=-8pt]{mutation} (node cs:name=a, anchor=south east);

\draw[dashed] (3,-0.5) -- (3,2.25);

\node at (3.5,0.44)   {$\mcO(1)$};
\node at (4.5,0.44)   {$\wedge^2 \mcU^{\vee}(1)$};
\node at (4,1.3)     {$\mcU^{\vee}(1)$};

\draw[dashed] (5.25,-0.5) -- (5.25,2.25);

\node at (6,0.87) {$\mcO(2)$};

\end{tikzpicture}
\end{center}

To obtain~$(2)$, we furthermore mutate $\mathbb{S}^2 \mcU^{\vee} (-1)$ to give $\mcO (-2)$, using the first sequence from Lemma~\ref{lem.seqs} twisted by $\mcO (-1)$. Then mutating $\mcO (-2)$ to the right through the collection gives
$\mcO (-2) \otimes \omega_{\Gr}^\vee  \cong \mcO (3)$,
where $\omega_{\Gr}\cong  \mcO(-5)$ is the canonical bundle of $\Gr(3,5)$.
\end{proof}

\section{Geometry of the flop} \label{sect:geom of flop}
This section further reviews the geometry of the varieties $X_+$ and $X_-$ from~\S\ref{sect:AG4flop}.
\subsection{GLSM/GIT construction}\label{sec:git}
In this section, we will characterize the varieties $X_\pm$ as GIT quotients with respect to a $\GL(3)$-action on a suitably-defined vector space. The argument has already been established in \cite{ourpaper_cy3s} and generalized in \cite{ourpaper_generalizedroofs}. We recall the main aspects of the construction.\\
\\
Let us fix $S\simeq\C^3$, $V\simeq \C^5$. Then, we define the following vector space:
\begin{equation*}
    W \coloneqq \Hom(S, V)\oplus \Hom(S, (\det S)^{\otimes 2}).
\end{equation*}
The fundamental action of $\GL(S)$ on $S$ defines an action on $W$:
\begin{align*}
        \GL(S)\times W & \to W \\
        (g, (B, \omega)) & \mapsto (B g^{-1}, \omega g^{-1}\det g^{2}),
\end{align*}
where the linear form $\omega\in S^\vee$ is treated as a row vector. Write $\mcX \coloneqq [W/\GL(S)]$ for the associated  Artin stack. For $\tau\in\Z$ one has the associated character:
\begin{align*}
        \rho_\tau\colon \GL(S) & \to \C^* \\
        g & \mapsto \det g^{-\tau},
\end{align*}
and for any $k = (k_1, k_2, k_3)\in\Z^3$ we can define a one-parameter subgroup (1-PS):
\begin{align*}
        g_k\colon \C^* & \to \GL(S) \\
        t  & \mapsto \diag(t^{k_1}, t^{k_2}, t^{k_3}).
\end{align*}
The GIT unstable locus with respect to $\tau$ can be characterized as the locus of $(B, \omega)\in W$ such that there exists a 1-PS $g_k$ which satisfies the following conditions:
\begin{enumerate}
    \item $\lim_{t\to 0} \rho_\tau(g_k(t))^{-1} = 0$
    \item $\lim_{t\to 0} g_k(t).(B, \omega)$ exists.
\end{enumerate}
There are essentially two stability conditions for the action of $G$ on $W$,
and they will be referred to using symbols $+$ and $-$.
Let $W_+^{\mathrm{ss}}$ and $W_-^{\mathrm{ss}}$ be the semistable locus with respect to each stability.

\begin{lem}
    Let us consider the vector space $W = \Hom(S,V)\oplus \Hom(S, (\det S)^{\otimes 2})$, with $\GL(S)$ acting as above.
    There exist two GIT quotients of $W$ with respect to $\GL(S)$:
    \begin{equation}
        \begin{split}
            X_+ &=  W_+^{\mathrm{ss}} / \GL(S)\simeq \Tot(\mcU_3^\vee(-2)\arw \Gr(3, V)),\\
            X_- &=  W_-^{\mathrm{ss}} / \GL(S)\simeq \Tot(\mcQ_3(-2)\arw \Gr(2, V)).
        \end{split}
    \end{equation}
\end{lem}
\begin{proof}
    See \cite[Lemma 4.4]{ourpaper_generalizedroofs} for $k = 2$.
\end{proof}
Note that both $X_-$ and $X_+$ embed in the Artin stack $\mcX = [W/\GL(S)]$. 
Let us denote the inclusions by $\iota_{\pm} \colon X_{\pm} \hookrightarrow \mcX$.

\subsection{The morphisms related to the flop} \label{sect:diagram flop}

With the definitions from \S\ref{sect:AG4flop},
by a little abuse of notations,
let the natural projections be as follows.
\[
\begin{tikzpicture}[auto,x=12mm,yscale=0.8]

\node (G+) at (0, 2) {$\Gr(3,5)$};
\node (G-) at (4, 2) {$\Gr(2,5)$};
\node (F) at (2, 2)  {$\Fl(2,3;5)$};
\node (X+) at (0,1) {$X_+$};
\node (X-) at (4,1) {$X_-$};
\node [transform canvas={yshift=+2.5pt}] (wX) at (2,1) {$\wX$};
\node (G0+) at (0,0) {$\G_+$};
\node (G0-) at (4,0) {$\G_-$};
\node [transform canvas={yshift=+1.5pt}] (E) at (2, 0) {$E$};

\draw[->] (wX) to node[swap] {$\scriptstyle p$} (X+);
\draw[->] (wX) to node {$\scriptstyle q$} (X-);

\draw[->] (F) to node[swap] {$\scriptstyle p$} (G+);
\draw[->] (F) to node {$\scriptstyle q$} (G-);

\draw[->] (E) to node[swap] {$\scriptstyle p$} (G0+);
\draw[->] (E) to node {$\scriptstyle q$} (G0-);

\end{tikzpicture}
\]
The following diagrams summarize all the morphisms we consider in relation to the flop.
Since there are too many morphisms to show in one diagram, 
we divide them into two.

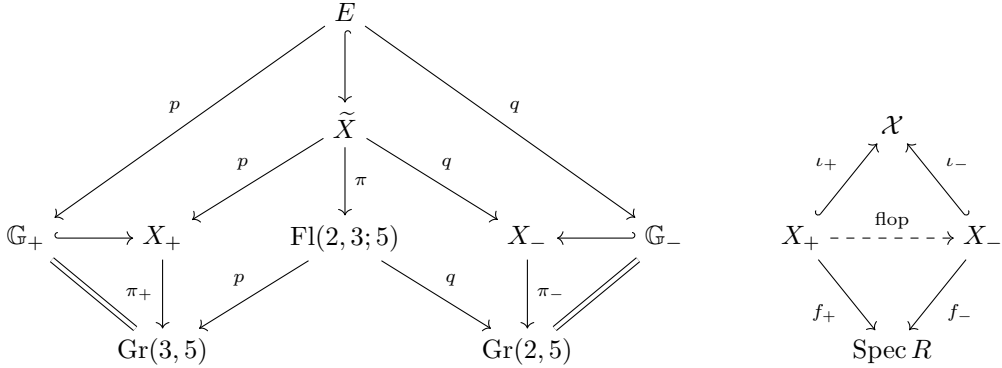
\[
\begin{tikzpicture}[auto,x=12mm]

\node (X+) at (0,0) {$X_+$};
\node (X-) at (4,0) {$X_-$};
\node (wX) at (2,1.5) {$\wX$};
\node (G+) at (0, -1.5) {$\Gr(3,5)$};
\node (G-) at (4, -1.5) {$\Gr(2,5)$};
\node (F) at (2, 0)  {$\Fl(2,3;5)$};
\node (G0+) at (-1.5,0) {$\G_+$};
\node (G0-) at (5.5,0) {$\G_-$};
\node (E) at (2, 3) {$E$};

\draw[->] (X+) to node[swap] {$\scriptstyle \pi_+$} (G+);
\draw[->] (X-) to node {$\scriptstyle \pi_-$} (G-);
\draw[->] (wX) to node {$\scriptstyle \pi$} (F);

\draw[{Hooks[right]}->] (G0+) -- (X+);
\draw[{Hooks[left]}->] (G0-) -- (X-);
\draw[{Hooks[right]}->] (E) -- (wX);

\draw[->] (wX) to node[swap] {$\scriptstyle p$} (X+);
\draw[->] (wX) to node {$\scriptstyle q$} (X-);

\draw[->] (F) to node[swap] {$\scriptstyle p$} (G+);
\draw[->] (F) to node {$\scriptstyle q$} (G-);

\draw[->] (E) to node[swap] {$\scriptstyle p$} (G0+);
\draw[->] (E) to node {$\scriptstyle q$} (G0-);

\draw[double equal sign distance] (G0+) -- (G+);
\draw[double equal sign distance] (G0-) -- (G-);

\node (X+2) at (7,0) {$X_+$};
\node (X-2) at (9,0) {$X_-$};
\node (cX) at (8, 1.5) {$\mcX$};
\node (R) at (8, -1.5) {$\Spec R$};

\draw[{Hooks[right]}->] (X+2) -- node {$\scriptstyle \iota_+$} (cX);
\draw[{Hooks[left]}->] (X-2) -- node[swap] {$\scriptstyle \iota_-$} (cX);

\draw[->] (X+2) to node[swap] {$\scriptstyle f_+$} (R);
\draw[->] (X-2) to node {$\scriptstyle f_-$} (R);

 \draw[dashed, ->] (X+2) -- node {$\scriptstyle \text{flop}$} (X-2);

\end{tikzpicture}
\]

Let $U$ be the \textit{common open subset}  of $X_+$ and $X_-$ with open embeddings $j_{\pm} \colon U \hookrightarrow X_{\pm}$ such that
\begin{enumerate}
\item[(1)] $j_{\pm}(U) = X_{\pm} \setminus \G_{\pm}$ and
\item[(2)] $f_+ \circ j_+ = f_- \circ j_-$.
\end{enumerate} 
Since $\codim_{X_{\pm}} \G_{\pm} = 3$,
there are equivalences of categories
\begin{align} \label{eq:reflexive equiv} 
\refl(X) \xleftarrow[j_+^*]{\sim} \refl(U) \xrightarrow[j_-^*]{\sim} \refl(X_-),
\end{align}
which will be frequently used in later sections.

\section{Construction of tilting bundles}\label{sec:windows}

The aim of this section is to construct tilting bundles over $X_{\pm}$,
which will be used to prove Theorem~\ref{thm:main_theorem}.

\subsection{The vector bundles}

In this section we introduce several important vector bundles over $X_{\pm}$,
and define the bundles that we will prove to be tilting.

\begin{nota}
Define vector bundles over $X_{\pm}$ by
\begin{center}
$\mcO_{X_{\pm}}(a) \coloneqq \pi_{\pm}^*\mcO_{\Gr}(a)$,
$\scrU_3 \coloneqq \pi^*_+\mcU_3$, and
$\scrU_2 \coloneqq \pi^*_- \mcU_2$.
\end{center}
In addition, define line bundles over $\wX$ by
\[ \mcO_{\wX}(aH_3 + bH_2) \coloneqq p^*\mcO_{X_+}(a) \otimes q^*\mcO_{X_-}(b). \]
Note that $\mcO_{\wX}(E) \simeq \mcO(-H_3-H_2)$.
\end{nota}

Note that one has $\scrU_3^\vee = \iota_+^*(W \times^{\GL(S)} S^\vee)$.
Moreover, there is the corresponding rank three vector bundle on $X_-$.
\begin{defi}
Define a rank three vector bundle $\scrP$ over $X_-$ by
\[ \scrP \coloneqq \iota_-^*\left(W \times^{\GL(S)} S^\vee \right). \]
\end{defi}
This bundle $\scrP$ is not a pullback from $\Gr(2, 5)$. In fact, one has the following:

\begin{lem} \label{lem:def P}
    There is a short exact sequence on $X_-$:
    \begin{equation} \label{eq: def P}
        0\arw \mcO_{X_-}(-2) \arw \scrP \arw \scrU_2^\vee \arw 0.
    \end{equation}
  In addition, if $(U, j_{\pm})$ is the common open subset,
   we have $j^*_-\scrP \simeq j_+^*\scrU_3^{\vee}$. 
\end{lem}

\begin{proof}
    The result is very similar to \cite[Equation 4.50]{ourpaper_generalizedroofs} and proven in exactly the same way, hence we will be brief. First, one observes that:
    \begin{equation*}
        X_- \simeq \{(B, \omega)\in W : \ker(B)\cap\ker(\omega) = 0, \omega\neq 0\}/ \GL(S).
    \end{equation*}
    This description can be simplified by observing that every point $(B, \omega)$ with $\ker(B) \cap \ker(\omega) = 0$ and $\omega\neq 0$ lies in the same orbit of $(B', (1,0,0))$ for some $B'$. Hence, we have
    \begin{equation*}
        X_- \simeq \{B\in \Hom(S, V) : \ker(B)\cap\ker((1,0,0)) = 0\}/ \Gamma
    \end{equation*}
    where $\Gamma$ is the stabilizer of the subspace $\{(B, (1,0,0))\}\subset W$, and it is explicitly given by:
    \begin{equation*}
        \Gamma = \left\{
            \left(
                \begin{array}{cc}
                    \lambda & x \\
                    0 & h
                \end{array}
            \right) : \lambda\in\C^*, h\in \GL(2), x = (x_1, x_2)\in\C^2
        \right\}.
    \end{equation*}
    Then, it is clear that we have an equivariant injection from $\mcO_{X_-}(-2)$ and an equivariant surjection to $\scrU_2^\vee$. The details are essentially the same as in \cite[page 2334]{ourpaper_generalizedroofs}.
    The stated isomorphism on $U$ is now clear from the construction.
\end{proof}

\begin{rem}
Throughout this paper we write $\Sym^2\scrP(1)$ for $(\Sym^2\scrP) \otimes \mcO(1)$, for example.
The reason for adopting this bracket-omitting notation, despite its potential ambiguity, 
is to avoid the excessive use of parentheses in the forthcoming formulas, 
which would otherwise obscure the intended structure of the expressions.
\end{rem}

\begin{defi} \label{def:tilting bundles}
Define four vector bundles over $X_+ = \Tot_{\Gr(3,5)}(\mcU_3^{\vee}(-2))$ as follows:
\begin{align*}
\scrT_+^{\spadesuit} &\coloneqq \mcO(-1) \oplus \mcO \oplus \mcO(1) \oplus \mcO(2) \oplus \mcO(3) \oplus \scrU_3^\vee(-1) \oplus \scrU_3^\vee \oplus \scrU_3^\vee(1) \oplus \scrU_3(1) \oplus \scrU_3(2) ,\\
\scrT_+^{\heartsuit} &\coloneqq \mcO(-1) \oplus \mcO \oplus \mcO(1) \oplus \mcO(2) \oplus \scrU_3^\vee(-1) \oplus \scrU_3^\vee \oplus \scrU_3^\vee(1) \oplus \scrU_3(1) \oplus \scrU_3(2) \oplus \Sym^2\scrU_3^\vee(-1), \\
\scrT_+^{\clubsuit} &\coloneqq (\scrT_+^{\spadesuit})^{\vee}, ~\textit{and} \\
\scrT_+^{\diamondsuit} &\coloneqq (\scrT_+^{\heartsuit})^{\vee}.
\end{align*}
Similarly, define four vector bundles over $X_- = \Tot_{\Gr(2,5)}(\mcQ_3(-2))$ as follows:
\begin{align*}
\scrT_-^{\spadesuit} &\coloneqq \mcO(1) \oplus \mcO \oplus \mcO(-1) \oplus \mcO(-2) \oplus \mcO(-3) \oplus \scrP(1) \oplus \scrP \oplus \scrP(-1) \oplus \scrP^{\vee}(-1) \oplus \scrP^{\vee}(-2) ,\\
\scrT_-^{\heartsuit} &\coloneqq \mcO(1) \oplus \mcO \oplus \mcO(-1) \oplus \mcO(-2) \oplus \scrP(1) \oplus \scrP \oplus \scrP(-1) \oplus \scrP^{\vee}(-1) \oplus \scrP^{\vee}(-2) \oplus \Sym^2\scrP(1), \\
\scrT_-^{\clubsuit} &\coloneqq (\scrT_-^{\spadesuit})^{\vee}, ~\textit{and} \\
\scrT_-^{\diamondsuit} &\coloneqq (\scrT_-^{\heartsuit})^{\vee}.
\end{align*}
\end{defi}

The vector bundles above with the same suit are related with each other as in the following lemma.

\begin{cor} \label{cor:same End}
For each $\star \in \{\spadesuit, \clubsuit, \heartsuit, \diamondsuit \}$,
and the common open subset $(U,j_{\pm})$,
there is an isomorphism
\[ j_+^* \scrT_+^{\star} \simeq j_-^*\scrT_-^{\star}. \]
This isomorphism induces an isomorphism of $R$-algebras
\[ \End_{X_+}(\scrT_+^{\star}) \simeq \End_{X_-}(\scrT_-^{\star}). \]
\end{cor}

\begin{proof}
The isomorphism of vector bundles can be deduced immediately from Lemma~\ref{lem:def P}.
Then the algebra isomorphism also follows from the equivalence (\ref{eq:reflexive equiv}).
\end{proof}

The rest of this section will be devoted to proving that all bundles in Definition~\ref{def:tilting bundles} are actually tilting.

\subsection{Tilting bundles over $X_+$} \label{sect:plus-side-tilting}
Let us begin with an elementary but useful fact on irreducible homogeneous bundles, which is well-known to experts:

\begin{lem}\label{lem:globally_generated}
    Let $\scrE$ be a globally generated, homogeneous, irreducible vector bundle on a rational homogeneous $Y$, fix $X \coloneqq \operatorname{Tot}(\scrE^\vee)$ and call $\pi \colon X \to Y$ the structure map. Then, for every homogeneous, irreducible vector bundle $\scrF$ on $Y$ which has a nontrivial global section, and for any $i>0$, one has $H^i(Y, \scrF) = H^i(X, \pi^*\scrF) = 0$.
\end{lem}
\begin{proof}
    If an irreducible homogeneous bundle $F$ has a nonzero global section, then it is globally generated, and hence it does not have higher cohomology because, by the Borel--Weil--Bott theorem, homogeneous irreducible vector bundles cannot have nontrivial cohomology in two different degrees. Since the tensor product of globally generated vector bundles is globally generated, the same holds for every irreducible summand of $\scrF\otimes\bigoplus_{l\geq 0}\Sym^l \scrE$: therefore $\scrF\otimes\bigoplus_{l\geq 0}\Sym^l \scrE$ has no higher cohomology. The proof is concluded by observing that the latter is no other than $\pi_*(\pi^*\scrF\otimes\mcO_X)$.
\end{proof}

\begin{thm}\label{lem:tilting_plus_side}
The bundle $\scrT_+^{\star}$ 
is a tilting bundle over $X_+$ for all $\star \in \{\spadesuit, \clubsuit, \heartsuit, \diamondsuit \}$.
\end{thm}

\begin{proof}
By Lemma~\ref{lem:dual tilting},
we can restrict our attention to the case $\star \in \{\spadesuit, \heartsuit\}$.

Let $\mcT_+^{\spadesuit}$ (resp.~$\mcT_+^{\heartsuit}$) be a vector bundle over $\Gr(3,5)$ that is the direct sum of the collection in Proposition~\ref{prop our FEC}~(1) (resp.~(2)).
Then notice that $\scrT_+^{\star} \simeq \pi_+^* \mcT_+^{\star}$ for each $\star \in \{\spadesuit, \heartsuit\}$.
Thus \cite[Lemma~3.1]{hara_G2} implies that the bundle $\scrT_+^{\star}$ classically generates $\mcD(X_+)$.
We are left to prove that $\Ext^p_{X_+}(\pi_+^*\scrT_+^{\star},\pi_+^*\scrT_+^{\star}) = 0$ for $p>0$. First, observe that:
    \begin{equation}\label{eq:vector_bundle_expansion}
        \begin{split}
            \Ext^p_{X_+}(\pi_+^*\mcT_+^{\star},\pi_+^*\mcT_+^{\ast}) & =  \Ext^p_{\Gr(3,5)}(\mcT_+^{\star},\pi_{+*}\pi_+^*\mcT_+^{\ast}) \\
            & =  \Ext^p_{\Gr(3,5)}\left(\mcT_+^{\star}, \mcT_+^{\star} \otimes \bigoplus_{l\geq 0} \Sym^l \mcU_3(2l) \right) \\
            & = \bigoplus_{l\geq 0} H^p(\Gr(3,5), \mcT_+^{\star\vee} \otimes \mcT_+^{\star}\otimes\Sym^l \mcU_3(2l))
        \end{split}
    \end{equation}
    Once we expand $\mcT_+^{\star}$, the task reduces to computing cohomology of homogeneous irreducible vector bundles on a Grassmannian.
    Note that, there is $l_0$ such that, in Equation \ref{eq:vector_bundle_expansion}, for every $l\geq l_0$ all summands appearing in the last line are globally generated, and thus have no higher cohomology. In fact, every such summand has the form
    \begin{equation*}
        \bS^{(a,b,0)}\mcU_3^\vee\otimes\mcO(-k)\otimes\Sym^l\mcU_3(2l) \simeq \bS^{(a,b,0)}\mcU_3^\vee\otimes\bS^{l,l}\mcU_3^\vee(l-k),
    \end{equation*}
    where $0\leq k\leq 4$. Hence, we have $l_0 = 4$.
    The computation for $l\leq l_0$ can be handled by means of the Littlewood--Richardson rule (c.f.~\cite{FultonHarrisBook}) and the Borel--Weil--Bott theorem (c.f.~\cite{weyman}). 
One can perform those computations by hand, or with a computational tool.
See \S\ref{code:lem:tilting_plus_side} for an example of a script that computes the relevant cohomologies.
\end{proof}

\subsection{Cohomology computations over $X_-$} \label{sect:minus-side-tilting}

The aim of the present section is to compute various extension groups of vector bundles over $X_-$,
which will be used to show the pretilting property of the bundle $\scrT_-^{\star}$ for $\star \in \{\spadesuit, \clubsuit, \heartsuit, \diamondsuit \}$.

\subsubsection{Easy Vanishings} \label{sect: easy ext}
In this section, we will extensively apply the Borel--Weil--Bott theorem to compute cohomology of vector bundles on $X_-$.
\begin{prop}\label{prop:easy_vanishings_line_bundles}
$H^i(X_-,\mcO(k)) = 0$ for all $k \geq -4$ and $i > 0$.
\end{prop}

\begin{proof}
For $k\geq 0$ the bundle $\mcO(k)$ on $\Gr(2, 5)$ has nontrivial global sections, thus the claim follows from Lemma \ref{lem:globally_generated}. For $-4\leq k\leq -1$, one can easily prove that $\mcO_{X_-}(k)$ has no higher cohomologies by means of the Borel--Weil--Bott theorem (see \S\ref{code:prop:easy_vanishings_line_bundles} for an example of a script that computes this cohomology).
\end{proof}

\begin{prop} \label{P}
$H^i(X_-,\scrP(k)) = 0$ for all $k \geq -2$ and $i > 0$.
\end{prop}

\begin{proof}
Consider the exact sequence 
$0 \to \mcO(-2+k) \to \scrP(k) \to \scrU_2^{\vee}(k) \to 0$. 
Then the first term has no higher cohomology by Proposition \ref{prop:easy_vanishings_line_bundles},
and likewise for the last one by the Littlewood-Richardson rule and the Borel--Weil--Bott theorem 
(see \S\ref{code:P} for an example of a script that computes this cohomology).
\end{proof}

\begin{prop} \label{sym2P}
$H^i(X_-,\Sym^2\scrP(k)) = 0$ for all $k \geq 0$ and $i > 0$.
\end{prop}

\begin{proof}
Taking the second symmetric product of (\ref{lem:def P}), and then twisting by $\mcO(k)$, 
we have an exact sequence
\[ 0 \to \scrP(-2+k) \to \Sym^2\scrP(k) \to \Sym^2\scrU_2^{\vee}(k) \to 0. \]
Then the result follows from by applying Proposition~\ref{P} and the Borel--Weil--Bott theorem.
(Note that $\Sym^2\mcU_2^\vee(k)$ is globally generated, 
and hence Lemma \ref{lem:globally_generated} applies.)
\end{proof}

\subsubsection{Key lemmas to prove the vanishing of problematic extensions}

The vanishings proved in the previous section are not enough to prove that the bundles $\scrT_-^{\star}$ are tilting.
In order to prove the remaining vanishing, we prepare several lemmas.

\begin{lem} \label{lem1}
There are isomorphisms $Rq_*p^*\scrU_3  \simeq \scrP^{\vee} \simeq Rq_*(p^*\scrU_3 \otimes O(E))$.
\end{lem}

\begin{proof}
Since $\wX$ is the total space of a line bundle over $\Fl(2,3;5)$, there is an exact sequence
\begin{align} \label{seq:xi}
 0 \to q^*\scrU_2 \to p^*\scrU_3 \to \mcO(-H_3+H_2) \to 0
\end{align}
on $\wX$. Note that $\mcO(-H+h) \simeq \mcO(E + 2h)$.
Consider the exact sequence
\[ 0 \to \mcO \to \mcO(E) \to \mcO_E(E) \to 0. \]
Since $q \colon E \to \G_-$ is a $\P^2$-bundle, it holds that $Rq_*\mcO_E(E) = 0$, and hence
 $Rq_*\mcO(E) \simeq \mcO_{X_-}$.
 Similarly, it also holds that $Rq_*\mcO(2E) \simeq \mcO_{X_-}$.
This together with the projection formula implies that 
$Rq_*\mcO(E + 2H_2) \simeq Rq_*\mcO(2E + 2H_2) \simeq \mcO(2)$ and $Rq_*(q^*\scrU_2 \otimes \mcO(E)) \simeq \scrU_2$.
Therefore, applying the functor $Rq_*$ to the sequence~\eqref{seq:xi} shows that
$Rq_*p^*\scrU_3$ is isomorphic to a vector bundle over $X_-$ and lies in an exact sequence
\begin{align} \label{seq:FMT for P}
0 \to \scrU_2 \to Rq_*p^*\scrU_3 \to \mcO(2) \to 0.
\end{align}
In particular, $Rq_*p^*\scrU_3$ is isomorphic to a vector bundle over $X_-$.
Note that, by Lemma~\ref{lem:def P}, the bundles $\scrP^{\vee}$ and $Rq_*p^*\scrU_3$ over $X_-$ both correspond to the bundle $\scrU_3$ over $X_+$ under the equivalence (\ref{eq:reflexive equiv}).
Therefore, there is an isomorphism 
$Rq_*p^*\scrU_3 \simeq \scrP^{\vee}$ as desired.
Similarly, applying the functor $Rq_*$ to the sequence~\eqref{seq:xi} tensored by $\mcO(E)$ proves the second desired isomorphism.
\end{proof}

\begin{lem} \label{P2}
There is an isomorphism $Rq_*p^*\scrU_3^{\vee}(E) \simeq \scrP$.
\end{lem}

\begin{proof}
Consider an exact sequence
\[ 0 \to \mcO(H_3-H_2) \to p^*\scrU_3^{\vee} \to q^*\scrU_2^{\vee} \to 0. \]
Since $\mcO(H_3-H_2) \simeq \mcO(-E-2H_2)$, twisting this sequence by $\mcO(E)$ gives
\[ 0 \to \mcO(-2H_2) \to p^*\scrU_3^{\vee}(E) \to q^*\scrU_2^{\vee}(E) \to 0. \]
Note that $Rq_*q^*\scrU_2^{\vee}(E) \simeq \scrU_2^{\vee} \otimes Rq_*\mcO(E) \simeq \scrU_2^{\vee}$.
Thus $Rq_*p^*\scrU_3^{\vee}(E) \simeq q_*p^*\scrU_3^{\vee}(E)$, and this sheaf fits in an exact sequence
\[ 0 \to \mcO_{X_-}(-2) \to q_*p^*\scrU_3^{\vee}(E) \to \scrU_2^{\vee} \to 0. \]
In particular, $Rq_*p^*\scrU_3^{\vee}(E)$ is isomorphic to a vector bundle over $X_-$.
Under the equivalence (\ref{eq:reflexive equiv}), 
both $\scrP$ and $Rq_*p^*\scrU_3^{\vee}(E)$ correspond to $\scrU_3^{\vee}$ by Lemma~\ref{lem:def P},
which proves that $Rq_*p^*\scrU_3^{\vee}(E) \simeq \scrP$.
\end{proof}

\begin{lem} \label{sym2 push}
There is an isomorphism $Rq_*p^* \Sym^2 \scrU_3 \simeq \Sym^2 \scrP^{\vee}$.
\end{lem}

\begin{proof}
Taking the second symmetric product of the exact sequence (\ref{seq:xi}) gives an exact sequence
\[ 0 \to q^* \Sym^2 \scrU_2 \to p^*\Sym^2 \scrU_3 \to p^*\scrU_3(-H_3+H_2) \to 0. \]
Note that $p^*\scrU_3(-H_3+H_2) \simeq p^*\scrU_3(E+2H_2)$,
and hence it follows from Lemma \ref{lem1} that
\[ Rq_*p^*\scrU_3(-H_3+H_2) \simeq \left( Rq_*p^*\scrU_3(E) \right)(2) \simeq \scrP^{\vee}(2). \]
Thus $Rq_*p^*\Sym^2 \scrU_3$ is a vector bundle over $X_-$ that lies in an exact sequence
\[ 0 \to \Sym^2 \scrU_2 \to Rq_*p^*\Sym^2 \scrU_3 \to \scrP^{\vee}(2) \to 0. \]
Note that, under the equivalence (\ref{eq:reflexive equiv}),
the bundles $Rq_*p^*\Sym^2 \scrU_3$ and $\Sym^2 \scrP^{\vee}$ over $X_-$ both correspond to the same bundle 
$\Sym^2 \scrU_3$ over $X_+$ by Lemma~\ref{lem:def P}.
Thus we have the desired isomorphism $Rq_*p^* \Sym^2 \scrU_3 \simeq \Sym^2 \scrP^{\vee}$.
\end{proof}

\begin{lem} \label{ex1}
There exists an exact sequence on $\wX$
\[ 0 \to q^*\scrP^{\vee} \to p^*\scrU_3 \to \mcO_E(E+2H_2) \to 0. \]
\end{lem}

\begin{proof}
The adjunction map $\mcO(2H_2) \simeq q^*q_*\mcO(E+2H_2) \to \mcO(E+2H_2)$ is an injective sheaf homomorphism with the cokernel $\mcO_E(E+2H_2)$.
Consider another adjunction map $q^*\scrP^{\vee} \simeq q^*q_*p^*\scrU_3 \to p^*\scrU_3$ from Lemma~\ref{lem1}, which is injective since $q$ is birational.
Now applying the snake lemma to the commutative diagram
\[ \begin{tikzcd}
0 \arrow[r] & q^*\scrU_2 \arrow[r] \arrow[d, equal] & q^*\scrP^{\vee} \arrow[r] \arrow[d] & \mcO(2H_2) \arrow[r] \arrow[d] & 0 \\
0 \arrow[r] & q^*\scrU_2 \arrow[r] & p^*\scrU_3 \arrow[r] & \mcO(E+2H_2) \arrow[r] & 0 
\end{tikzcd} \]
proves the result.
\end{proof}

\begin{lem} \label{ex2}
There exists an exact sequence on $\wX$
\[ 0 \to q^*\scrP \to p^*\scrU_3^{\vee}(E) \to q^*\scrU_2^{\vee}(E)|_E \to 0. \]
\end{lem}

\begin{proof}
By Lemma~\ref{P2}, there exists the adjunction morphism $q^*\scrP \to p^*\scrU_3^{\vee}(E)$, 
that lies in the following diagram of exact sequences
\[ \begin{tikzcd}
0 \arrow[r] & \mcO(-2H_2) \arrow[r] \arrow[d, equal] & q^*\scrP \arrow[d] \arrow[r] & q^*\scrU_2^{\vee} \arrow[d] \arrow[r] & 0 \\
0 \arrow[r] & \mcO(-2H_2) \arrow[r] & p^*\scrU_3^{\vee}(E) \arrow[r] & q^*\scrU_2^{\vee}(E) \arrow[r] & 0.
\end{tikzcd} \] 
The third vertical map is also the adjunction of the isomorphism $Rq_*q^*\scrU_2^{\vee}(E) \simeq  \scrU_2^{\vee}$,
and the cokernel is $q^*\scrU_2^{\vee}(E)|_E$.
Thus the snake lemma shows the result.
\end{proof}

\subsubsection{Proof of harder vanishings} \label{sect: prob ext}

\begin{prop} \label{Pv}
$H^i(X_-,\scrP^{\vee}(-a)) = 0$ for all $a \leq 4$ and $i > 0$.
\end{prop}

\begin{proof}
If $a \leq 1$, the Borel--Weil--Bott theorem applied to
    \begin{equation*}
    0 \to \scrU_2(-a) \to \scrP^{\vee}(-a) \to \mcO_{X_-}(2-a) \to 0
    \end{equation*} 
    proves the vanishing (as, under the identification $\scrU_2^\vee\simeq \scrU_2(1)$, the relevant vanishings have already been computed in Propositions \ref{prop:easy_vanishings_line_bundles} and \ref{P}).
Thus we may assume $a \in \{2,3, 4\}$.

Using Lemma \ref{lem1} gives
\begin{align*}
H^i(X_-, \scrP^{\vee}(-a)) &\simeq H^i(X_-, Rq_*p^*\scrU_3(-a)) \\
&\simeq H^i(\wX, p^*\scrU_3(-aH_2)) \\
&\simeq H^i(\wX, p^*\scrU_3(aE + aH_3)) \\
&\simeq H^i(X_+, \scrU_3(a) \otimes Rp_*\mcO(aE)).
\end{align*}
Consider the case $a =2$.
As in the proof of Lemma~\ref{P2},
it follows that $Rp_*\mcO_{\wX}(2E) \simeq \mcO_{X_+}$.
Therefore the vanishing
\begin{center} 
$H^{i}(X_-, \scrP^{\vee}(-2)) \simeq H^{i}(X_+, \scrU_3(2)) = 0$
\end{center}
follows for all $i > 0$, by applying Lemma \ref{lem:globally_generated} to $\scrU_3(2)\simeq \wedge^2\scrU_3^\vee(1)$.

Next, consider the case $a = 3$.
Note that the line bundle $\mcO(3E)$ lies in an exact sequence
\[ 0 \to \mcO(2E) \to \mcO(3E) \to \mcO_E(3E) \to 0. \]
Since $p \colon E \to \G_+$ is a $\P^2$-bundle, 
it follows that 
$Rp_*\mcO_E(3E) \simeq \omega_{\G_+}[-2] = \mcO_{\G_+}(-5)[-2]$.
Therefore, the complex $\scrU_3(3) \otimes Rp_*\mcO(3E)$ lies in an exact triangle
\[ \scrU_3(3) \to \scrU_3(3) \otimes Rp_*\mcO(3E) \to \scrU_3(-2)|_{\G_+}[-2]. \]
The following hold.
\begin{enumerate}
\item[(1)] $H^{\geq 1}(X_+, \scrU_3(3)) = 0$.
\item[(2)] The bundle $\scrU_3(-2)|_{\G_+} \simeq \mcU_3(-2)$ over $\G_+ \simeq \Gr(3,5)$ is acyclic.
\end{enumerate}
Indeed, the first claim is true by Lemma~\ref{lem:globally_generated}, while the latter because 
\[ \RG(\Gr(3,5), \mcU_3(-2)) \simeq \RHom(\mcU_3^{\vee}(1), \mcO(-1)) = 0 \]
by Proposition~\ref{prop our FEC}.
Thus the vanishing
\[ H^{i}(X_-, \scrP^{\vee}(-3)) \simeq H^{i}(X_+, \scrU_3(3) \otimes Rp_*\mcO(3E)) = 0 \]
also follows for all $i > 0$.

It remains to consider the case $a = 4$.
Applying the functor $Rp_*$ to the exact sequence $0 \to \mcO(3E) \to \mcO(4E) \to \mcO_E(4E) \to 0$ gives
an exact sequence
\[ 0 \to \mcO_{\G_+}(-5) \to R^2p_*\mcO(4E) \to \scrU_3^{\vee}(-7)|_{\G_+} \to 0 \]
on $X_+$.
Put $\mcF \coloneqq R^2p_*\mcO(4E)$.
Then the complex $\scrU_3(4) \otimes Rp_*\mcO(4E)$ lies in an exact triangle
\[ \scrU_3(4) \to \scrU_3(4) \otimes Rp_*\mcO(4E) \to \scrU_3(4) \otimes_{X_+} \mcF[-2]. \]
By the Borel--Weil--Bott theorem, the following hold.
\begin{enumerate}
\item[(3)] $H^{i}(X_+, \scrU_3(4)) = 0$ for all $i > 0$.
\item[(4)] The bundles 
\begin{align*}
\scrU_3(4)|_{\G_+} \otimes \mcO_{\G_+}\!(-5) & \simeq \mcU_3(-1) \\
\scrU_3(4)|_{\G_+} \otimes \mcU_3^{\vee}(-7) & \simeq \mcU_3 \otimes \mcU_3^{\vee}(-3)
\end{align*}
over $\G_+ \simeq \Gr(3,5)$ are acyclic.
\end{enumerate}
Indeed, (3) follows from Lemma \ref{lem:globally_generated}.
For the first bundle in (4), we have
\[ \RG(\Gr(3,5), \mcU_3(-1)) \simeq \RHom(\mcU_3^{\vee}(1), \mcO) = 0 \]
 by the semi-orthogonality of the collection in Proposition~\ref{prop our FEC}.
For the second bundle in (4), one can compute as
\begin{align*}
H^i(\Gr(3,5), \mcU_3 \otimes \mcU_3^{\vee}(-3)) &\simeq \Ext^i(\mcU_3^{\vee}(4), \mcU_3^{\vee}(1)) \\
&\simeq \Ext^{6-i}(\mcU_3^{\vee}(1), \mcU_3^{\vee}(-1))^*,
\end{align*}
and this is zero for all $i$, again by the semi-orthogonality of the collection in Proposition~\ref{prop our FEC}.

Now (4) implies that $\RG(\scrU_3(4)|_{\G_+} \otimes \mcF[-2]) = 0$, and hence
\begin{align*}
H^{i}(X_-, \scrP^{\vee}(-4)) &\simeq H^{i}(X_+, \scrU_3(4) \otimes Rp_*\mcO(4E)) \\
&\simeq H^{i}(X_+, \scrU_3(4)) \\
&= 0
\end{align*}
for all $i>0$, as desired.
\end{proof}

\begin{prop} \label{sym2}
$H^i(X_-, \Sym^2 \scrP^{\vee} (-k)) = 0$ for all $k \in \{0, 1,2,3\}$ and $i>0$.
\end{prop}

\begin{proof}
Let $k \in \{0, 1,2,3\}$. Then using Lemma~\ref{sym2 push} gives
\begin{align*}
H^i(X_-, \Sym^2 \scrP^{\vee} (-k)) &\simeq H^i\left(X_-, \left(Rq_*p^* \Sym^2 \scrU_3\right) (-k)\right) \\
&\simeq H^i\left(\wX,p^* \Sym^2 \scrU_3 (-kH_2)\right) \\
&\simeq H^i\left(\wX,p^* \Sym^2 \scrU_3 (kE+kH_3)\right) \\
&\simeq H^i\left(\wX,p^* \Sym^2 \scrU_3 (kE+kH_3)\right) \\
&\simeq H^i\left(X_+, \Sym^2 \scrU_3(k) \otimes Rp_*\mcO(kE) \right).
\end{align*}
Therefore there is an isomorphism
\begin{align*}
H^i(X_-, \Sym^2 \scrP^{\vee}(-k)) &\simeq H^i(X_+, \Sym^2 \scrU_3(k))
\end{align*}
for all $0 \leq k \leq 2$.
In addition, there is an exact triangle
\[ \Sym^2 \scrU_3(3) \to \Sym^2 \scrU_3(3) \otimes Rp_*\mcO(3E) \to \Sym^2 \scrU_3(-2)|_{\G_+}[-2]. \]
Thus, in order to show the vanishing of $H^{\geq 1}(X_-, \Sym^2 \scrP^{\vee} (-k))$, it is enough to prove the following two statements.
\begin{enumerate}
\item[(1)]  $H^{i}(X_+, (\Sym^2 \scrU_3)(k)) = 0$ for all $i \geq 1$ and $k \in \{0,1,2,3\}$.
\item[(2)] The bundle $\Sym^2 \scrU_3(-2)|_{\G_+} \simeq \Sym^2 \mcU_3(-2)$ over $\G_+ \simeq \Gr(3,5)$ is acyclic.
\end{enumerate}
Note that, the vanishings in (1) have already been addressed in Theorem~\ref{lem:tilting_plus_side}.
For (2), note that
\[ H^j(\Gr(3,5), \Sym^2 \mcU_3(-2)) \simeq \Ext_{\Gr(3,5)}^{6-j}(\mcO(2), \Sym^2 \mcU_3^{\vee}(-1))^* \]
by Serre duality, and this extension is indeed zero by the semi-orthogonality of the collection in Proposition~\ref{prop our FEC}~(2).
\end{proof}

\begin{prop} \label{PvP}
$\Ext^i_{X_-}(\scrP^{\vee}, \scrP(k)) = 0$ for all $k \geq 0$ and $i > 0$.
\end{prop}

\begin{proof}
Since 
\[ \Ext^i_{X_-}(\scrP^{\vee}, \scrP(k)) \simeq H^i(X_-, \scrP \otimes \scrP(k)) 
\simeq H^i\left(X_-, \Sym^2 \scrP(k) \oplus \scrP^{\vee}(k) \right), \]
the vanishing follows from Proposition~\ref{sym2P} and \ref{Pv}.
\end{proof}

\begin{prop}
$\Ext^i_{X_-}(\scrP, \scrP^{\vee}(-a)) = 0$ for all $a \in \{0,1,2,3\}$ and $i > 0$.
\end{prop}

\begin{proof}
Since 
\[ \Ext^i_{X_-}(\scrP, \scrP^{\vee}(-a)) \simeq H^i(X_-, \scrP^{\vee} \otimes \scrP^{\vee}(-a)) 
\simeq H^i\left(X_-, \Sym^2 \scrP^{\vee}(-a) \oplus \scrP(-a) \right), \]
the vanishing follows from Proposition~\ref{P} and \ref{sym2}.
\end{proof}

\begin{prop} \label{prop Pv and Pv}
$\Ext_{X_-}^i(\scrP^{\vee}, \scrP^{\vee}(-a)) \simeq \Ext_{X_-}^i(\scrP, \scrP(-a)) = 0$ for all $a \leq 2$ and $i>0$.
\end{prop}

\begin{proof}
First, let us prove the case when $a < 0$.
Then Lemma~\ref{lem:globally_generated} shows that
\begin{enumerate}
\item[(i)] $\Ext_{X_-}^i(\scrU_2^{\vee}, \scrU_2^{\vee}(-a)) \simeq H^i(X_-, \Sym^2\scrU_2^{\vee}(-a-1) \oplus \mcO(-a)) = 0$ for all $i>0$, and
\item[(ii)] $\Ext_{X_-}^i(\mcO(-2), \scrU_2^{\vee}(-a)) \simeq H^i(X_-,\scrU_2^{\vee}(-a+2)) = 0$ for all $i > 0$.
\end{enumerate}
Thus, applying the functor $\RHom_{X_-}(-, \scrU_2^{\vee}(-a))$ to the sequence (\ref{eq: def P})
and then using (i) and (ii) shows that
\begin{enumerate}
\item[(iii)] $\Ext_{X_-}^i(\scrP, \scrU_2^{\vee}(-a)) = 0$ for all $i > 0$.
\end{enumerate}
Again, applying the functor $\RHom(\scrP, - \otimes \mcO(-a))$ to (\ref{eq: def P}) gives an exact sequence
\[ H^i(X_-, \scrP^{\vee}(-a-2)) \to \Ext_{X_-}^i(\scrP, \scrP(-a)) \to \Ext_{X_-}^i(\scrP, \scrU_2^{\vee}(-a)). \]
For any $i > 0$, the first term of this sequence is zero by Proposition~\ref{Pv},
and the third term is zero by (iii).
Therefore the desired vanishing $\Ext_{X_-}^i(\scrP, \scrP(-a)) = 0$ holds for all $a < 0$ and $i>0$.

It remains to prove the cases when $a = 0, 1, 2$.
Apply the functor $\RHom_{\wX}(-,p^*\scrU_3(-aH_2))$ to the exact sequence in Lemma~\ref{ex1}
to obtain the following exact triangle. 
\begin{align*}
\RHom_{\wX}(\mcO_E(E+2H_2), p^*\scrU_3(-aH_2)) & \to \\\RHom_{\wX}(p^*\scrU_3,p^*\scrU_3(-aH_2)) & \to \\ \RHom_{\wX}(q^*\scrP^{\vee},p^*\scrU_3(-aH_2)) &
\end{align*}
Then the third term of this triangle can be computed as
\[ \RHom(q^*\scrP^{\vee},p^*\scrU_3(-aH_2)) \simeq \RHom(\scrP^{\vee},Rq_*p^*\scrU_3(-aH_2)) \simeq \RHom(\scrP^{\vee},\scrP^{\vee}(-a)), \]
and this is what we want to compute.
The cohomology of the first term can be computed as
\begin{align*}
&\Ext_{\wX}^i( \mcO_E(E+2H_2),p^*\scrU_3(-aH_2)) \\
{}\simeq{} &\Ext_{\wX}^{9-i}(p^*\scrU_3(-aH_2), \mcO_E(3E+2H_2))^* &\text{(Serre duality)} \\
{}\simeq{} &H^{9-i}\left(\Fl(2,3;5), \mcU_3^{\vee}((1-a)E+(-2-a)H_3)\right)^* & \\
{}\simeq{} &H^{9-i}\left(\Gr(3,5), \mcU_3^{\vee} \otimes \Sym^{a-1} \mcU_3 \otimes \mcO(a-4) \right)^*.
\end{align*}
This is clearly zero if $a = 0$.
In addition, this cohomology is isomorphic to $\Ext^{9-i}_{\Gr(3,5)}(\mcO(3),\mcU_3^{\vee})$ if $a = 1$,
and to $\Ext^{9-i}_{\Gr(3,5)}(\mcU_3^{\vee}(1),\mcU_3^{\vee}(-1))$ if $a = 2$.
These are zero for all $i$ by the semi-orthogonality of the collection in Proposition~\ref{prop our FEC}.
Thus the first term $\RHom( \mcO_E(E+2H_2),p^*\mcU_3(-aH_2))$ is zero.

%

Finally, the second term is 
\[ \RHom_{\wX}(p^*\scrU_3,p^*\scrU_3(-aH_2)) \simeq \RHom_{X_+}(\scrU_3, \scrU_3(a) \otimes Rp_*\mcO(aE)) \simeq \RHom_{X_+}(\scrU_3^{\vee}, \scrU_3^{\vee}(a)), \]
which has no cohomology in non-zero degrees for all $a \in \{0,1,2\}$ 
due to the pretilting property of the bundle $\scrT_+^{\spadesuit}$ proved in Theorem~\ref{lem:tilting_plus_side}.

Therefore the desired vanishing $\Ext^i(\scrP^{\vee}, \scrP^{\vee}(-a)) = 0$ holds for all $a \in \{0,1,2\}$ and $i>0$.
\end{proof}

\begin{prop} \label{prop P sym2P}
$\Ext_{X_-}^i(\scrP, \Sym^2\scrP(a)) = 0$ for all $a \in \{0,1,2\}$ and $i>0$.
\end{prop}

\begin{proof}
Taking the second symmetric product of the exact sequence 
$0 \to \mcO(-2) \to \scrP \to \scrU_2^{\vee} \to 0$ gives another sequence
\[ 0 \to \scrP(-2) \to \Sym^2 \scrP \to \Sym^2 \scrU_2^{\vee} \to 0. \]
Let $a \in \{0,1,2\}$.
Applying $\Ext_{X_-}^i(\scrP(-a), - )$ to the above sequence gives an exact sequence
\[ \Ext_{X_-}^i(\scrP, \scrP(a-2)) \to \Ext_{X_-}^i(\scrP, \Sym^2 \scrP(a)) \to \Ext_{X_-}^i(\scrP, \Sym^2 \scrU_2^{\vee}(a)). \]
Proposition~\ref{prop Pv and Pv} proves that the first term is zero for all $i>0$.
The third term fits in an exact sequence
\[ \Ext_{X_-}^i(\scrU_2^{\vee}, \Sym^2 \scrU_2^{\vee}(a)) \to \Ext_{X_-}^i(\scrP, \Sym^2 \scrU_2^{\vee}(a)) \to H^i(X_-, \Sym^2 \scrU_2^{\vee}(a+2)). \]
The rightmost term has no higher cohomology by Lemma \ref{lem:globally_generated}, 
while the cohomology of the leftmost one can be computed by the Littlewood-Richardson rule and the Borel--Weil--Bott theorem
(see \S\ref{code:prop P sym2P} for an example of a script for this computation).
\end{proof}

\begin{prop}
$\Ext_{X_-}^i(\scrP^{\vee}, \Sym^2\scrP(a)) = 0$ for all $a \in \{2,3\}$ and $i>0$.
\end{prop}

\begin{proof}
Applying the functor $\RHom(\scrP^{\vee}(-a),-)$ to the sequence
\[ 0 \to \scrP(-2) \to \Sym^2 \scrP \to \Sym^2 \scrU_2^{\vee} \to 0 \]
gives an exact triangle
\[ \RHom(\scrP^{\vee},\scrP(-2+a)) \to \RHom(\scrP^{\vee},\Sym^2 \scrP(a)) \to \RHom(\scrP^{\vee},\Sym^2 \scrU_2^{\vee}(a)). \]
The first term has no higher cohomology by Proposition~\ref{PvP}.
Thus it is enough to show that the third term has no cohomology.
Applying the functor $\RHom(-, \Sym^2\scrU_2^{\vee}(a))$ to an exact sequence
\[ 0 \to \scrU_2 \to \scrP^{\vee} \to \mcO(2) \to 0 \]
gives an exact triangle
\[ \RG(X_-, \Sym^2\scrU_2^{\vee}(a-2)) \to \RHom(\scrP^{\vee},\Sym^2 \scrU_2^{\vee}(a)) \to \RHom(\scrU_2,\Sym^2 \scrU_2^{\vee}(a)). \]
The left term has no higher cohomology by Lemma \ref{lem:globally_generated}.
The right term is
\[ \RHom(\scrU_2,\Sym^2 \scrU_2^{\vee}(a)) \simeq \RG(X_-, \Sym^3\scrU_2^{\vee}(a) \oplus \scrU_2^{\vee}(a+1)), \]
and this has no higher cohomology for the same reason.
\end{proof}

\begin{prop} \label{Sym2P&P}
$\Ext^i(\Sym^2\scrP, \scrP(-a)) = 0$ for all $a \in \{0,1,2\}$ and $i>0$.
\end{prop}

\begin{proof}
Consider the exact sequence $0 \to q^*\scrP^{\vee} \to p^*\scrU_3 \to \mcO_E(E+2H_2) \to 0$ given in Lemma~\ref{ex1},
and apply the functor $\Ext^i(-,p^*\Sym^2\scrU_3(-aH_2))$.
Then the first term is
\begin{align*}
\Ext_{\wX}^i(\mcO_E(E+2H_2),p^*\Sym^2\scrU_3(-aH_2)) &\simeq  \Ext^{9-i}(p^*\Sym^2\scrU_3(-aH_2), \mcO_E(3E+2H_2))^* \\
&\simeq H^{9-i}(\Fl(2,3;5), p^*\Sym^2\mcU_3^{\vee}((1-a)E-(a+2)H_3))^* \\
&\simeq H^{9-i}(\Gr(3,5), \Sym^2\mcU_3^{\vee} \otimes \Sym^{a-1}\mcU_3 \otimes \mcO(a-4))^*.
\end{align*}
If $a = 0$, then this is clearly zero.
In addition, the above cohomology is isomorphic to $\Ext^{9-i}_{\Gr(3,5)}(\mcU_3^{\vee}(1), \Sym^2\mcU_3^{\vee}(-1))$ if $a = 1$,
and to $\Ext^{9-i}_{\Gr(3,5)}(\mcO(2), \Sym^2\mcU_3^{\vee}(-1))$ if $a = 2$.
These extensions are zero for all $i \in \Z$, by the semi-orthogonality of the collection in Proposition~\ref{prop our FEC}~(2).

The second term is
\[ \Ext_{\wX}^i(p^*\scrU_3,p^*\Sym^2\scrU_3(-aH_2)) \simeq \Ext^i_{X_+}(\scrU_3, \Sym^2\scrU_3(a) \otimes Rp_*\mcO_{\wX}(aE)). \]
Since $Rp_*\mcO_{\wX}(aE) \simeq \mcO_{X_+}$ for all $a \in \{0,1,2\}$, this extension is isomorphic to 
\[ \Ext^i_{X_+}(\scrU_3, \Sym^2\scrU_3(a)) \simeq \Ext^i_{X_+}(\Sym^2\scrU_3^{\vee}(-1), \scrU_3^{\vee}(a-1)). \]
This extension is zero for all $a \in \{0,1,2\}$ and $i >0$
by the pretilting property of the bundle $\scrT_+^{\heartsuit}$ proved in Theorem~\ref{lem:tilting_plus_side}.


The third term can be computed as
\begin{align*}
\Ext_{\wX}^i(q^*\scrP^{\vee},p^*\Sym^2\scrU_3(-aH_2)) &\simeq \Ext_{\wX}^i(\scrP^{\vee},Rq_*p^*\Sym^2\scrU_3 \otimes \mcO(-a)) \\
&\simeq \Ext_{\wX}^i(\scrP^{\vee}, \Sym^2\scrP^{\vee}(-a))
\end{align*}
by Lemma~\ref{sym2 push}, which is what we want to compute.
Now the vanishings of the first and second terms show the vanishing of the third term.
\end{proof}

\begin{prop} \label{prop:sym2P-sym2P}
$\Ext^i(\Sym^2\scrP, \Sym^2 \scrP) = 0$ for all $i>0$.
\end{prop}

\begin{proof}
Apply the functor $\Ext^i(\Sym^2\scrP, -)$ to an exact sequence
\[ 0 \to \scrP(-2) \to \Sym^2\scrP \to \Sym^2 \scrU_2^{\vee} \to 0. \]
Then the second term is what we want to compute,
and the first term is zero for all $i > 0$ by Proposition~\ref{Sym2P&P}.

Let us compute the third term 
\[ \Ext^i(\Sym^2\scrP, \Sym^2 \scrU_2^{\vee}) \simeq H^i(X_-, \Sym^2\scrP^\vee \otimes \Sym^2 \scrU_2^{\vee}). \] 
Consider the following two exact sequences
\begin{align*}
&0 \to \scrP(-2) \to \Sym^2\scrP \to \Sym^2 \scrU_2^{\vee} \to 0,~ \text{and} \\
&0 \to \scrU_2(-2) \to \scrP(-2) \to \mcO \to 0,
\end{align*}
which shows that the bundle $\Sym^2\scrP^\vee \otimes \Sym^2 \scrU_2^{\vee}$ can be filtered by
bundles 
\begin{center}
$\Sym^2\scrU_2^\vee(2)$, $\Sym^2\scrU_2^\vee \otimes \scrU_2^\vee(4)$ and 
$\Sym^2\scrU_2\otimes \Sym^2\scrU_2^\vee$.
\end{center}
Then, the bundles $\Sym^2\scrU_2^\vee(2)$ and $\Sym^2\scrU_2^\vee \otimes \scrU_2^\vee(4)$ have no higher cohomologies by Lemma \ref{lem:globally_generated}, 
while the vanishing of higher cohomologies of the third bundle $\Sym^2\scrU_2\otimes \Sym^2\scrU_2^\vee$ follows by the Littlewood-Richardson rule and the Borel--Weil--Bott theorem (see \S\ref{code:prop:sym2P-sym2P} for an example of a script that computes this cohomology).
Thus $ \Ext^i(\Sym^2\scrP, \Sym^2 \scrU_2^{\vee}) = 0$ for all $i > 0$.
\end{proof}

\begin{prop} \label{prop:sym2P-Pv}
$\Ext^i(\Sym^2 \scrP, \scrP^{\vee}(-a)) = 0$ for all $a \in \{2,3\}$ and $i>0$.
\end{prop}

\begin{proof}
Apply the functor $\Ext^i_{\wX}(-, p^*\Sym^2 \scrU_3(-aH_2))$ to the exact sequence in Lemma~\ref{ex2}.
Then the third term is
\begin{align*}
\Ext^i_{\wX}(q^*\scrP, p^*\Sym^2 \scrU_3(-aH_2)) &\simeq \Ext^i_{X_-}(\scrP, Rq_*p^*\Sym^2 \scrU_3(-a)) \\
&\simeq \Ext^i_{X_-}(\scrP, \Sym^2 \scrP^{\vee}(-a)) \\
&\simeq \Ext^i_{X_-}(\Sym^2 \scrP, \scrP^{\vee}(-a))
\end{align*}
by Lemma~\ref{sym2 push}, which is what we want to compute.

The second term is 
\begin{align*}
\Ext^i_{\wX}(p^*\scrU_3^{\vee}(E), p^*\Sym^2 \scrU_3(-aH_2)) &\simeq \Ext^i_{\wX}(p^*\scrU_3^{\vee}, p^*\Sym^2 \scrU_3((a-1)E+aH_3)) \\
&\simeq \Ext^i_{X_+}(\scrU_3^{\vee}, \Sym^2 \scrU_3(a) \otimes Rp_*\mcO((a-1)E)) \\
&\simeq \Ext^i_{X_+}(\scrU_3^{\vee}, \Sym^2 \scrU_3(a)) \\
&\simeq \Ext^i_{X_+}(\Sym^2 \scrU_3^{\vee}(-1), \scrU_3(a-1))
\end{align*}
Here the third isomorphism follows since $Rp_*\mcO((a-1)E) \simeq \mcO_{X_+}$ for $a \in \{2,3\}$.
Now this extension is zero for all $i > 0$ by the pretilting property of the bundle $\scrT_+^{\heartsuit}$ proved in Theorem~\ref{lem:tilting_plus_side}.

Therefore it is enough to prove that the first term has no higher cohomologies.
It can be computed as
\begin{align*}
\Ext^i_{\wX}(q^*\scrU_2^{\vee}(E)|_E, p^*\Sym^2 \scrU_3(-aH_2)) &\simeq \Ext^{9-i}_{\wX}(p^*\Sym^2 \scrU_3(-aH_2), q^*\scrU_2^{\vee}(3E)|_E)^* \\
&\simeq H^{9-i}(\Fl(2,3;5), \Sym^2 \mcU^{\vee}_3 \otimes \mcU_2^{\vee} \otimes \mcO_E(3E+aH_2))^*.
\end{align*}
Consider an exact sequence
\[ 0 \to \mcO(H_3-H_2) \to \mcU_3^{\vee} \to \mcU_2^{\vee} \to 0 \]
on $E \simeq \Fl(2,3;5)$ and then twist this sequence by $\Sym^2 \mcU^{\vee}_3 \otimes \mcO_E(3E+aH_2)$.
Then the cohomology exact sequence is
\begin{align*} H^{j}(\Fl(2,3;5), \Sym^2 \mcU^{\vee}_3(2E+(a-2)H_2)) &\to H^{j}(\Fl(2,3;5), \Sym^2 \mcU^{\vee}_3 \otimes \mcU_3^{\vee} \otimes \mcO_E(3E+aH_2)) \\
&\to H^{j}(\Fl(2,3;5), \Sym^2 \mcU^{\vee}_3 \otimes \mcU_2^{\vee}\otimes \mcO_E(3E+aH_2)).
\end{align*}
The first term of this sequence can be computed as
\begin{align*}
H^{j}(\Fl(2,3;5), \Sym^2 \mcU^{\vee}_3(2E+(a-2)H_2)) &\simeq H^{j}(\Fl(2,3;5), \Sym^2 \mcU^{\vee}_3((4-a)E+(2-a)H_3))\\
&\simeq H^j(\Gr(3,5), \Sym^2 \mcU^{\vee}_3(2-a) \otimes Rp_*\mcO_E((4-a)E)).
\end{align*}
This is zero for all $j \geq 0$, since for $a \in \{2,3\}$, one has $4-a \in \{1,2\}$, and hence $Rp_*\mcO_E((4-a)E) = 0$.

The second term of the cohomology long exact sequence is
\begin{align*}
&H^{j}(\Fl(2,3;5), \Sym^2 \mcU^{\vee}_3 \otimes \mcU_3^{\vee} \otimes \mcO_E(3E+aH_2)) \\
{}\simeq{} &H^{j}(\Fl(2,3;5), \Sym^2 \mcU^{\vee}_3 \otimes \mcU_3^{\vee} \otimes \mcO_E((3-a)E-aH_3)) \\
{}\simeq{} &H^{j}(\Gr(3,5), \Sym^2 \mcU^{\vee}_3 \otimes \mcU_3^{\vee}(-a) \otimes Rp_*\mcO_E((3-a)E)).
\end{align*}
If $a = 2$, this is zero for all $j$ since $Rp_*\mcO_E(E) = 0$.
If $a=3$, this is 
\[ H^{j}(\Gr(3,5), \Sym^2 \mcU^{\vee}_3 \otimes \mcU_3^{\vee} \otimes \mcO(-3))
\simeq \Ext^j_{\Gr(3,5)}(\mcU_3(2), \Sym^2\mcU_3^{\vee}(-1)),
 \]
 which is zero for all $j$ by the semi-orthogonality of the collection in Proposition~\ref{prop our FEC}~(2).
\end{proof}

\subsection{Tilting bundles over $X_-$}

\begin{thm}
The bundle $\scrT_-^{\star}$ is a tilting bundle over $X_-$ for all $\star \in \{\spadesuit, \clubsuit, \heartsuit, \diamondsuit \}$.
\end{thm}

\begin{proof}
Cohomology computations in \S\ref{sect: easy ext} and \S\ref{sect: prob ext} show that $\scrT_-^{\spadesuit}$ and $\scrT_-^{\heartsuit}$ are pretilting.
Since $\End_{X_-}(\scrT_-^{\star}) \simeq \End_{X_+}(\scrT_+^{\star})$, 
this algebra has finite global dimension for all $\star \in \{\spadesuit, \clubsuit, \heartsuit, \diamondsuit \}$.
Therefore applying Proposition~\ref{generation criteria} shows that $\scrT_-^{\star}$ is a generator, and hence a tilting bundle.
\end{proof}

\begin{rem}
The restriction of the bundle $\scrP$ to the zero-section $\Gr(2,5) \simeq \G_- \subset X_-$ is isomorphic to $\mcO(-2) \oplus \mcU_2^{\vee} \simeq \mcO(-2) \oplus \mcU_2(1)$.
Thus the restriction of the tilting bundle $\scrT_-^{\spadesuit}$ to the zero-section is isomorphic to a direct sum (with multiplicities) of the following collection of bundles over $\Gr(2,5)$:
\[ \{\mcO(a) \mid -3 \leq a \leq 1\} \cup \{\mcU_2(b) \mid -2 \leq b \leq 2\}, \]
and this set forms the full strong exceptional collection
\begin{align} \label{Lef exc coll Gr(2,5)}
 \mcD(\Gr(2,5)) = \langle \mcO(-3), \mcU_2(-2), \mcO(-2), \mcU_2(-1), \mcO(-1), \mcU_2, \mcO, \mcU_2(1), \mcO(1), \mcU_2(2) \rangle
\end{align}
of $\Gr(2,5)$.
In other words, the restriction $\scrT_-^{\spadesuit}$ remains tilting, and is given by a full strong exceptional collection.

On the contrary, the restriction of the tilting bundle $\scrT_-^{\heartsuit}$ to the zero-section is isomorphic to a direct sum (with multiplicities) of the following collection of bundles over $\Gr(2,5)$:
\[ \{\mcO(a) \mid -3 \leq a \leq 1\} \cup \{\mcU_2(b) \mid -2 \leq b \leq 2\} \cup \{\Sym^2 \mcU_2(3) \}. \]
The first two sets of bundles actually form the same full exceptional collection (\ref{Lef exc coll Gr(2,5)}).
However, as an indecomposable summand of $\scrT_-^{\heartsuit}|_{\mathbb{G}_-}$, there appears an additional bundle, namely $\Sym^2 \mcU_2(3)$.
In particular, the restriction $\scrT_-^{\heartsuit}|_{\mathbb{G}_-}$ is no longer tilting.

This observation indicates that tilting bundles are not necessarily constructed solely from full strong exceptional collections on the base.
This would highlight the difficulty of constructing a tilting bundle in more general settings including simple flops of type $A^G_n$ for $n \geq 6$.
\end{rem}

\begin{rem} \label{rem:why not Kapranov}
Pulling-back the direct sum of Kapranov's collection over $\Gr(3,5)$ to $X_+$ gives the bundle
\[ \mcO \oplus \mcO(1) \oplus \mcO(2) \oplus \scrU_3^{\vee} \oplus \scrU_3^{\vee}(1) \oplus \scrU_3(1) \oplus \scrU_3(2) \oplus \Sym^2 \scrU_3^{\vee} \oplus \bS^{2,1}\scrU_3^{\vee} \oplus \bS^{2,2}\scrU_3^{\vee}. \]
One can show that this bundle is also tilting (c.f.~\S\ref{sect:explicitHLwindow}).
The corresponding bundle over $X_-$ (in the sense of Corollary~\ref{cor:same End}) is 
\[ \mcO \oplus \mcO(-1) \oplus \mcO(-2) \oplus \scrP \oplus \scrP(-1) \oplus \scrP^{\vee}(-1) \oplus \scrP^{\vee}(-2) \oplus \Sym^2 \scrP \oplus \bS^{2,1}\scrP \oplus \bS^{2,2} \scrP. \]
Unfortunately, this bundle is not pretilting.
Indeed, mimicking the computations in \S\ref{sect:minus-side-tilting},
\begin{center}
$\Ext^3(\scrP^{\vee}(-1), \Sym^2\scrP) \simeq \Ext^6(\scrP^{\vee}(-1), \Sym^2\scrP) \simeq \C$.
\end{center}
This is the reason why we should work with bundles in Definition~\ref{def:tilting bundles},
which correspond to the new exceptional collections in \S\ref{sect:ourFEC}.
\end{rem}

\section{Conclusions} \label{sect:conclusion}

\subsection{The magic window categories} \label{sect:window main result}

\begin{defi}
\begin{enumerate}
\item[(1)] For each $\star \in \{\spadesuit, \clubsuit, \heartsuit, \diamondsuit \}$, 
define the set $\scrC^{\star}$ of dominant $\GL(S)$-weights as follows.
\begin{align*}
\scrC^{\spadesuit} 
  &\coloneqq 
   \left\{
   \begin{array}{l}
     (-1,-1,-1), (0,0,0), (1,1,1), (2,2,2), (3,3,3), \\
     (0,-1,-1), (1,0,0), (2,1,1), (1,1,0), (2,2,1)
   \end{array}
   \right\}, \\
\scrC^{\heartsuit} 
  &\coloneqq 
   \left\{
   \begin{array}{l}
     (-1,-1,-1), (0,0,0), (1,1,1), (2,2,2), \\
     (0,-1,-1), (1,0,0), (2,1,1), (1,1,0), (2,2,1), (1,-1,-1)
   \end{array}
   \right\}, \\
\scrC^{\clubsuit} 
  &\coloneqq 
   \left\{
   \begin{array}{l}
     (-3,-3,-3), (-2,-2,-2), (-1,-1,-1), (0,0,0), (1,1,1), \\
     (-1,-1,-2), (0,0,-1), (1,1,0), (-1,-2,-2), (0,-1,-1)
   \end{array}
   \right\}, \\
   \scrC^{\diamondsuit} 
  &\coloneqq 
   \left\{
   \begin{array}{l}
     (-2,-2,-2), (-1,-1,-1), (0,0,0), (1,1,1), \\
     (-1,-1,-2), (0,0,-1), (1,1,0), (-1,-2,-2), (0,-1,-1), (1,1,-1)
   \end{array}
   \right\}.
\end{align*}

\item[(2)] For each $\star \in \{\spadesuit, \clubsuit, \heartsuit, \diamondsuit \}$,
let $\scrT^{\star}$ be a vector bundle over $\mcX$ that is defined as
\[ \scrT^{\star} \coloneqq \bigoplus_{\chi \in \scrC^{\star}} \mcO_W \otimes V(\chi), \]
where $V(\chi)$ is the irreducible $\GL(S)$-representation whose highest weight is $\chi$.

\item[(3)] For each $\star \in \{\spadesuit, \clubsuit, \heartsuit, \diamondsuit \}$, let $\scrW^{\star} \subset \mcD(\mcX)$
denote the thick subcategory generated by the bundle $\scrT^{\star}$.
In addition, define the \textit{restriction functors} $\res_{\pm}^{\star}$ as the composite
\[ \res^{\star}_{\pm} \colon \scrW^{\star} \hookrightarrow \mcD(\mcX) \xrightarrow{\iota_{\pm}^*} \mcD(X_{\pm}).\]
\end{enumerate}
\end{defi}

\begin{rem}
Note that $\iota_+^*(\mcO_W \otimes V(\chi)) \simeq \bS^{\chi}\scrU_3^{\vee}$
and $\iota_-^*(\mcO_W \otimes V(\chi)) \simeq \bS^{\chi}\scrP$.
\end{rem}

\begin{defi}
For $\star \in \{\spadesuit, \clubsuit, \heartsuit, \diamondsuit \}$,
let $\mcT_{\pm}^{\star}$ be the bundle in Definition~\ref{def:tilting bundles}.
Then define an $R$-algebra $\Lambda^{\star}$ by
\[ \Lambda^{\star} \coloneqq \End_{X_+}(\scrT_+^{\star}) \simeq \End_{X_-}(\scrT_-^{\star}). \]
\end{defi}

\begin{thm} \label{thm:DHKR window}
The restriction functors 
$\res^{\star}_{\pm} \colon \scrW^{\star} \to \mcD(X_{\pm})$
are equivalences of categories for each  $\star \in \{\spadesuit, \clubsuit, \heartsuit, \diamondsuit \}$.
\end{thm}

\begin{proof}
Fix $\star \in \{\spadesuit, \clubsuit, \heartsuit, \diamondsuit \}$.
By construction $\res^{\star}_{\pm}(\scrT^{\star}) \simeq \scrT_{\pm}^{\star}$.
In addition, the natural morphism
\[ \Ext^i_{\scrW^{\star}}(\scrT^{\star}, \scrT^{\star}) \to \Ext_{X_{\pm}}^i(\scrT_{\pm}^{\star}, \scrT_{\pm}^{\star}) \]
is an isomorphism 
for all $i \neq 0$ since the bundles $\scrT^{\star}$ and $\scrT_{\pm}^{\star}$ enjoy the pretilting property,
and for $i = 0$ since $\codim_W (W \setminus W^{\mathrm{ss}}_{\pm}) \geq 2$.
Since the bundle $\scrT^{\star}$ is a classical generator of the category $\scrW^{\star}$ by definition,
Proposition~\ref{prop:generator ff} implies that the functors $\res^{\star}_{\pm}$ are fully faithful.

Since the derived category $\mcD(\mcX)$ is idempotent complete, so is the thick subcategory $\scrW^{\star}$.
Therefore Proposition~\ref{prop: dense im} shows that the essential image of $\res^{\star}_{\pm}$ is a thick subcategory of $\mcD(X_{\pm})$.
Note that the essential image of $\res^{\star}_{\pm}$ contains the bundle $\scrT_{\pm}^{\star}$, which is a classical generator of $\mcD(X_{\pm})$.
Thus the essential image coincides with the whole category $\mcD(X_{\pm})$, and hence $\res^{\star}_{\pm}$ is essentially surjective.
\end{proof}

With the property proved in Theorem~\ref{thm:DHKR window},
the subcategory $\scrW^{\star}$ is called a \textit{window} for both open subschemes $X_+, X_- \subset \mcX$.
Furthermore, by definition, the window $\scrW^{\star}$ is generated by sheaves of the form $\mcO \otimes V(\chi)$.
A window subcategory of $\mcD(\mcX)$ that enjoys this property is sometimes called a \textit{magic} window category (c.f.~\cite{halpernleistner_derived}).

\subsection{The main results}

\begin{cor} \label{cor:main result}
For each $\star \in \{\spadesuit, \clubsuit, \heartsuit, \diamondsuit \}$, the following hold.
\begin{enumerate}
\item[\rm (1)]  
The functor 
\[ \res_-^{\star} \circ (\res_+^{\star})^{-1}  \colon \mcD( X_+) \xrightarrow{\sim} \mcD( X_-) \]
 is an equivalence of $R$-linear triangulated categories such that $\Phi(\mcO_{X_+}) \simeq \mcO_{X_-}$.
 \item[\rm (2)] The following diagram commutes.
 \[ \begin{tikzpicture}[auto,x=12mm]
 \node (W) at (0,0) {$\scrW^{\star}$};
 \node (+) at (-2,-1) {$\mcD(X_+)$};
 \node (-) at (2,-1) {$\mcD(X_-)$}; 
 \node (A) at (0,-2) {$\mcD(\Lambda^{\star})$};
 
 \draw[->] (W) to node[swap] {$\scriptstyle \res_+^{\star}$} (+);
 \draw[->] (W) to node {$\scriptstyle \res_-^{\star}$} (-);
 \draw[->] (+) to node[swap] {$\scriptstyle \RHom(\scrT_+^{\star},-)$} (A);
 \draw[->] (-) to node {$\scriptstyle \RHom(\scrT_-^{\star},-)$} (A);
 \end{tikzpicture} \]
\item[\rm (3)] The algebra $\Lambda^{\star}$ is an NCCR of $R$ and derived equivalent to both crepant resolutions $X_{\pm}$ of $\Spec R$.
\end{enumerate}
\end{cor}

\begin{proof}
(1) is a direct consequence of Theorem~\ref{thm:DHKR window}.
Note that $\mcO_W \in \scrW^{\star}$ and $\res_{\pm}(\scrO_W) \simeq \mcO_{X_{\pm}}$.
The commutativity of the diagram in (2) follows since by construction there are functorial isomorphisms
$\RHom_{\scrW^{\star}}(\scrT^{\star},-) \simeq \RHom_{X_{\pm}}(\scrT_{\pm}^{\star}, \res_{\pm}^{\star}(-))$. 
Finally~(3) follows from Proposition~\ref{tilting prop}.
\end{proof}

\subsection{Application to Calabi--Yau threefolds} \label{sect:CY3results}

Recall that there is a natural isomorphism $H^0(\Gr(3,5), \mcU_3(2)) \simeq H^0(\Gr(2,5), \mcQ_3^{\vee}(2))$.
Let $s_+ \in H^0(\Gr(3,5), \mcU_3(2))$ be a section, and $s_- \in H^0(\Gr(2,5), \mcQ_3^{\vee}(2))$ the corresponding section.
If $s_{\pm}$ are regular sections, the zero loci $\bV(s_+) \subset \Gr(3,5)$ and $\bV(s_-) \subset \Gr(2,5)$ are (possibly singular) projective Calabi--Yau threefolds.
This pair $(\bV(s_+), \bV(s_-))$ is called a pair of dual Calabi--Yau threefolds in \cite{ourpaper_cy3s}.

The following corollary was first proved in \cite[Theorem~5.7]{ourpaper_cy3s} when the Calabi--Yau threefolds $\bV(s_{\pm})$ are smooth.
The following gives a new proof via the flop and noncommutative algebra, and extends the result of \cite{ourpaper_cy3s} to singular situations.

\begin{cor} \label{cor:CY3Deq} 
Let $s_+ \in H^0(\Gr(3,5), \mcU_3(2))$ be a section, 
and $s_- \in H^0(\Gr(2,5), \mcQ_3^{\vee}(2))$ the corresponding section.
If $s_{\pm}$ are regular sections,
then there is a derived equivalence
\[ \mcD(\bV(s_+)) \simeq \mcD(\bV(s_-)) \]
of (possibly singular) Calabi--Yau threefolds.
\end{cor}

\begin{proof}
Let $\Gm$ acts on $\Hom(S,V)$ trivially, and on $\Hom(S, (\det S)^{\otimes 2})$ by scaling. 
This induces a $\Gm$-action on $W \coloneqq \Hom(S, V) \oplus \Hom(S, (\det S)^{\otimes 2})$,
which descends to a $\Gm$-action on $\mcX \coloneqq [W/\GL(S)]$.
In addition, for each stability, the stable locus is $\Gm$-invariant, and hence the above action also induces 
$\Gm$-actions on $X_+$ and $X_-$, which coincides with the fiber-wise scaling action.

Note that all bundles over $\mcX$ of the form $\mcO \otimes V(\chi)$ are $\Gm$-linearizable.
This implies that, for any $\star \in \{\spadesuit, \clubsuit, \heartsuit, \diamondsuit \}$,
the tilting bundles $\scrT_+^{\star}$ and $\scrT_-^{\star}$ admit $\Gm$-equivariant structure such that the natural isomorphism $\End_{X_+}(\scrT_+^{\star}) \simeq \End_{X_-}(\scrT_-^{\star})$ is $\Gm$-equivariant.
Now by \cite[Theorem~1.2]{hiranoknorrer} and \cite[Theorem~1.2]{hiranoequivtilt} (see also \cite{shipman, Isik13}), there are equivalences
\begin{align*}
\mcD(\bV(s_+)) &\simeq \Dcoh_{\Gm}(X_+, \chi_{\id}, Q_{s_+}) \simeq \Dmod_{\Gm}(\Lambda_+^{\star}, \chi_{\id}, Q_{\bar{s}}), \\
\mcD(\bV(s_-)) &\simeq \Dcoh_{\Gm}(X_-, \chi_{\id}, Q_{s_-}) \simeq \Dmod_{\Gm}(\Lambda_-^{\star}, \chi_{\id}, Q_{\bar{s}}),
\end{align*}
where the notations are as follows.
\begin{enumerate}
\item[$\bullet$] $\chi_{\id} \colon \Gm \to \Gm$ is the identity map considered as a character.
\item[$\bullet$] $Q_{s_{\pm}} \colon X_{\pm} \to \A^1$ are the functions corresponding to $s_{\pm}$.
\item[$\bullet$] $Q_{\bar{s}} \colon \Spec R \to \A^1$ is the function such that $Q_{s_{\pm}} = Q_{\bar{s}} \circ f_{\pm}$.
\item[$\bullet$] $\Lambda_{\pm}^{\star} = \End_{X_{\pm}}(\scrT_{\pm}^{\star})$.
\item[$\bullet$] $\Dcoh_{\Gm}(X_{\pm}, \chi_{\id}, Q_{s_{\pm}})$ is the derived factorization category associated to the gauged linear sigma model $(X_{\pm}, \chi_{\id}, Q_{s_{\pm}})^{\Gm}$.
\item[$\bullet$] $\Dmod_{\Gm}(\Lambda_{\pm}^{\star}, \chi_{\id}, Q_{\bar{s}})$ is the derived factorization category associated to the noncommutative gauged linear sigma model $(\Lambda_{\pm}^{\star}, \chi_{\id}, Q_{\bar{s}})^{\Gm}$.
\end{enumerate}
Thus, using that the isomorphism 
$\Lambda_+^{\star} \simeq \Lambda_-^{\star}$ in Corollary~\ref{cor:same End}
is $\Gm$-equivariant, we have the result.
\end{proof}

\section{Comparison with known window categories} \label{sect:comparisonHL}
The aim of this section is to explain why
our construction of tilting bundles cannot be deduced from
known window categories, for instance those coming from general results of Halpern-Leistner \cite{halpernleistner_GIT}.
We first explicitly describe the Kempf-Ness stratification for each stability condition in our setting, 
together with the associated one-parameter subgroups. 
We then examine Halpern-Leistner's window categories in detail,
and observe that the window categories we construct in \S\ref{sec:windows} and \S\ref{sect:window main result} differ from his window categories.

\subsection{Kempf-Ness stratification} We use the notation from~\S\ref{sec:git}. 
Consider a maximal subtorus $T$ of $\GL(S)=\GL(3)$ given by diagonal matrices with respect to a fixed basis of~$S$. 
Fixing a basis of $V\simeq \mathbb C^5$ we get natural coordinates of the space $W$ represented by  $(B,\omega)$, where $B=\left (B_{i,j}\right)$ is a $3\times 5$ matrix with entries $B_{i,j}$ for $i\in  \{1\dots 3\}$ and $j \in \{1\dots 5\}$, whereas $\omega=(\omega_1,\omega_2,\omega_3)$.  The action of the torus $T$ on $W$ equipped with the chosen basis is the diagonal action with the following weights
\begin{align*}& w(B_{1,j})=(-1,0,0) , \
 w(B_{2,j})=(0,-1,0) , \ 
 w(B_{3,j})=(0,0,-1)  \text{ for } j\in \{1\dots 5\} \\
 &w(\omega_1)=(1,2,2),\, w(\omega_2)=(2,1,2),\, w(\omega_3)=(2,2,1).
\end{align*}
We now fix a norm  the standard Euclidean norm so that  for $\lambda$ a 1-PS of $T$
\begin{align*}
        \lambda\colon \C^* & \to \GL(S) \\
        t  & \mapsto \diag(t^{k_1}, t^{k_2}, t^{k_3})
\end{align*}
we have $\|\lambda\|:=\sqrt{k_1^2+k_2^2+k_3^2}$. The norm is invariant with respect to the symmetric group. Moreover for a  character $\rho$ of $T$ corresponding to a weight $(r_1,r_2,r_3)$ we define $(\rho,\lambda)=r_1k_1+r_2k_2+r_3k_3 $.
Now, for any  point $v\in W$ we recall the definition 
$$M^{\rho}_T(v):= \inf_{\lambda} \frac{(\rho, \lambda)}{\| \lambda \|},$$
where the infimum is taken over all 1-PSs $\lambda$ for which $\lim_{t\to 0}\lambda (t)\cdot v$ exists.

We consider the unstable locus of the induced action of $T$ on $W$. By the Hilbert-Mumford criterion we have $v$ is unstable if $M^{\rho}_T(v) <0$ and $\lambda$ is called $\rho$-adapted to $v$ when $$M^{\rho}_T(v)=  \frac{(\rho, \lambda)}{\| \lambda \|}\text{ and }\lim_{t\to 0}\lambda (t)\cdot v \text{ exists. }$$

We stratify the unstable locus of the action of $T$ by their $\rho$-adapted 1-PSs for $\rho$ induced by characters of $\GL(S)$. By a classical result of Kempf adapted to the affine case (see \cite[Theorem 2.15]{Hoskins}), the Kempf-Ness stratification with respect to a character $\rho$ of the action of $\GL(S)$ will then be given by the $\GL(S)$-invariant subsets generated by the strata of the stratification of the unstable locus of the $T$-action with respect to the induced character of~$T$. Note that some of these invariant subsets might be equal for two different strata.

Note first that the condition ``$\lim_{t\to 0}\lambda (t)\cdot v$ exists'' is equivalent to the condition 
$(\lambda, w)\geq 0$ for each weight  $w$  corresponding to a coordinate at which $v\in W$ has nonzero coefficient.

Up to scaling we essentially have two cases for the character $\rho$. These are represented by $(r_1,r_2,r_3)=(1,1,1)$ or $(r_1,r_2,r_3)=(-1,-1,-1)$ which correspond to the two sides of our variation of GIT.

\subsubsection{The $+$ side}
When $(r_1,r_2,r_3)=(-1,-1,-1)$ for $v=(v_1\dots v_{18})=(B,\omega )\in W$  we are looking for triples $(k_1,k_2,k_3)$  minimising the  function 
$$ f(k_1,k_2,k_3):=-\frac{k_1+k_2+k_3}{\sqrt{k_1^2+k_2^2+k_3^2}}$$
under the conditions 
\begin{equation}\label{condition}
w(v_i) \cdot (k_1,k_2,k_3) \geq 0 \text{ for all  $i$ such that } v_i \neq 0.
\end{equation}
Note that to have $v$ unstable we need  $\frac{-k_1-k_2-k_3}{\sqrt{k_1^2+k_2^2+k_3^2}} <0$ for some $(k_1,k_2,k_3)$ satisfying condition \eqref{condition}. This means in particular that for $v$ to be unstable the corresponding condition \eqref{condition} must allow some $k_i$ to be positive. We now make two observations:
\begin{enumerate}
\item[(a)]  Changing some $k_i$ to $|k_i|$ decreases $f$.
\item[(b)] Changing a negative $k_i$ to zero decreases $f$.
\end{enumerate}
Now we have three possibilities for $(k_1,k_2,k_3)$ depending on $v$ or more precisely on the conditions \eqref{condition}.
\begin{enumerate}
\item All $k_i$ are allowed to be positive, which correspond to the case $B=0$, i.e.~$v=(0,\omega)$. Then the minimum is attained when $k_1=k_2=k_3>0$.  Indeed, by observation~(a),  to minimise $f$ we can assume that all $k_i$ are positive and use the inequality between the quadratic and arithmetic mean:
$$\frac{|k_1|+|k_2|+|k_3|}{3}\leq \sqrt{\frac{k_1^2+k_2^2+k_3^2}{3}}$$
to see that $f(k_1,k_2,k_3)\geq -\sqrt{3}$ and equality holds exactly when $k_1=k_2=k_3>0$. We hence get in this case $M(v)=-\sqrt 3$ and adapted subgroup with weight $(1,1,1)$.
\item Only two among  $k_i$ are allowed to be positive, say $k_1$ and $k_2$. Then, by observation~(b), we can set  $k_3=0$ and apply again the argument above for two variables to get $f(k_1,k_2,k_3)\geq -\sqrt{2}$ with equality holding exactly when $k_1=k_2>0$ and $k_3=0$. We get $M(v)=-\sqrt{2}$ and adapted subgroup with weights $(1,1,0)$.
\item Only one among $k_i$ is allowed to be positive, say $k_1$. As above we set $k_2=k_3=0$ and get $M(v)= -1$ with adapted subgroup with weight $(1,0,0)$.
\end{enumerate}
\subsubsection{The $-$ side}
When $(r_1,r_2,r_3)=(1,1,1)$ for $v=(B,\omega)\in W$  we are looking for triples $(k_1,k_2,k_3)$  minimising  the function 
$$ f(k_1,k_2,k_3):=\frac{k_1+k_2+k_3}{\sqrt{k_1^2+k_2^2+k_3^2}}$$
under the same conditions \eqref{condition}.

In this case to have $v$ unstable we need  $\frac{k_1+k_2+k_3}{\sqrt{k_1^2+k_2^2+k_3^2}} <0$ for some $(k_1,k_2,k_3)$ satisfying condition \eqref{condition}. This means that the condition \eqref{condition} must allow a solution with $k_1+k_2+k_3<0$.  In particular, for $v$ to be unstable, at least one among its three coordinates 
$\omega_1$, $\omega_2$, $\omega_3$ must be equal to $0$. Indeed, note that when  $\omega_1\neq 0$ condition \eqref{condition} implies 
 $2(k_1+k_2+k_3)-k_1\geq 0$ which together with $k_1+k_2+k_3<0$ implies $k_1<0$. Similarly $\omega_2\neq 0$ implies $k_2<0$ and   $\omega_3 \neq 0$ implies $k_3<0$, while all $k_i$ negative contradict $2k_1+2k_2+k_3\geq 0$.
 
Since we can always rescale the weight of the 1-PS by any positive multiple without changing the conditions as well as the 1-PSs we will assume $k_1+k_2+k_3=-1$ and minimizing $f$ will amount to minimizing the denominator $\sqrt{k_1^2+k_2^2+k_3^2}$ or equivalently minimizing $k_1^2+k_2^2+k_3^2$.

Now we have the following possibilities depending on $v$. 
\begin{enumerate}
\item All coordinates $\omega_1$, $\omega_2$, $\omega_3$ are zero. Then any triple $(k_1,k_2,k_3)$ of negative numbers satisfies \eqref{condition}. In this case the minimum is attained when $k_1=k_2=k_3<0$.  Indeed, observing that changing $k_i$ to $-|k_i|$ decreases $f$,  to minimise $f$ we can assume that all $k_i$ are not positive and use as before the inequality between the quadratic and arithmetic mean
to see that $f(k_1,k_2,k_3)\geq -\sqrt{3}$ and equality holds exactly when $k_1=k_2=k_3<0$. We hence get in this case $M(v)=-\sqrt 3$ and adapted subgroup with weight $(-\frac{1}{3},-\frac{1}{3},-\frac{1}{3})$ which after rescaling to integers gives $(-1,-1,-1)$.

\item Exactly one of the coordinates $\omega_1$, $\omega_2$, $\omega_3$ is nonzero, say $\omega_3$.  Then the condition \eqref{condition} implies $2(k_1+k_2)+k_3=2(k_1+k_2+k_3)-k_3=-2-k_3\geq 0$ and hence $k_3 \leq -2$, while $k_1+k_2\geq 1$.
We then have two possibilities:
\begin{enumerate}
\item $v$ via \eqref{condition} gives no further conditions on $k_1$ and $k_2$ (all coordinates of $v$ with weights $(-1,0,0)$ and $(0,-1,0)$ vanish) then taking into account $k_3 \leq -2$ and  $k_1+k_2+k_3=-1$ we get  $$k_1^2+k_2^2+k_3^2\geq \frac{1}{2}(k_1+k_2)^2+k_3^2\geq \frac{1}{2}(1+k_3)^2+k_3^2\geq \frac{9}{2}$$ and we have equality exactly when 
$k_1=k_2=\frac {1}{2}$ and $k_3=-2$. Hence $M(v)=-\frac{\sqrt{2}}{3}$ and the adapted subgroup has weight rescaled to integers $(1,1,-4)$.

\item $v$ via \eqref{condition} forces one among $k_1$, $k_2$ to be at most $0$ (all the coordinates of $v$ with weights $(-1,0,0)$ vanish or all the coordinates of $v$ with weights $(-1,0,0)$ vanish but not all coordinates of both types vanish), say $k_2$. Then $k_1\geq -1-k_3\geq 1$, which implies 
$$k_1^2+k_2^2+k_3^2\geq 1+4=5$$
and equality holds exactly for $(k_1,k_2,k_3)= (1,0,-2)$. In this case $M(v)=-\frac{1}{\sqrt{5}}$ and the adapted subgroup has weight $(1,0,-2)$.

\end{enumerate}

\item Exactly two among  the coordinates $\omega_1$, $\omega_2$, $\omega_3$ are nonzero, say $\omega_2$ and $\omega_3$. Then as above we have $k_3\leq -2$, $k_2\leq -2$, $k_1\geq 3$ which implies 
$$k_1^2+k_2^2+k_3^2\leq 4+4+9=17$$ and equality holds for $(k_1,k_2,k_3)= (3,-2,-2)$. Then $M(v)=-\frac{1}{\sqrt{17}}$ and the adapted 1-PS has weight  $(3,-2,-2)$.

\end{enumerate}
In all above we made some choices that differ by permutation of coordinates which gives in total $1+3+6+3=13$ strata of the stratification of the unstable locus with respect to the chosen maximal subtorus $T$. However acting by the group $\GL(S)$ on those strata we get the Kempf-Ness stratification of the unstable locus of the $\GL(S)$ action to be described as follows in terms of $v=(B,\omega)$.
\begin{enumerate}
\item Stratum given $\omega=0$ with $M((B,\omega))=-\sqrt 3$ and adapted subgroup with weight $(-1,-1,-1)$.
\item Stratum given by $\omega \neq 0$ and $\dim (\ker \omega \cap \ker B)=2$ with $M(v)=-\frac{\sqrt{2}}{3}$ and adapted subgroup with weight $(1,1,-4)$.
\item Stratum given by $\omega \neq 0$ and $\dim (\ker \omega \cap \ker B)=1$ with $M((B,\omega))=-\frac{1}{\sqrt{5}}$ and adapted subgroup with weight $(1,0,-2)$.
\end{enumerate}
Note that the type of unstable locus with $M(v)=-\frac{1}{\sqrt{17}}$ and weight $(1,0,-2)$ does not occur in the list above as such point $v$ can be moved by $\GL(S)$ to points in the stratum of the $T$ action with  $M(v)=-\frac{1}{\sqrt{5}}$ and adapted subgroup with weight $(1,0,-2)$.

\subsection{Known window categories} \label{sect:explicitHLwindow}

Following \cite[Definition~2.8]{halpernleistner_GIT}, for each element $w = (w_0,w_1,w_2) \in \Z^3$,
define subcategories $\scrG_+^w, \scrG_-^w \subset \mcD(\mcX)$.
By \cite[Theorem~2.10]{halpernleistner_GIT}, the restriction functors 
\begin{center}
$\res_+^w \colon \scrG_{+}^w \hookrightarrow \mcD([W/G]) \xrightarrow{\iota_+^*} \mcD(X_{+})$
and $\res_-^w \colon \scrG_{-}^w \hookrightarrow \mcD([W/G]) \xrightarrow{\iota_-^*} \mcD(X_{-})$
\end{center}
are equivalences.
In other words, the category $ \scrG_{+}^w$ (resp.~$\scrG_{-}^w$) is a window for $X_+ \subset \mcX$ (resp.~$X_- \subset \mcX$).
The categories $\scrG_+^w$ and $\scrG^w_-$ are also known as \textit{graded restriction categories}.

For each $+$ and $-$, 
define the set $\scrC_{\pm}^w$ of dominant $\GL(S)$-weights by
\[ \scrC_{\pm}^w \coloneqq \{ \chi \mid \mcO_W \otimes V(\chi) \in \scrG_{\pm}^w \}. \]
It follows from the equivalences $\res_{\pm}^w$ that the bundles
\begin{center}
$\iota_{+}^*\left(\bigoplus_{\chi \in \scrC_{+}^w} \mcO_W \otimes V(\chi) \right)$ over $X_+$ and
$\iota_{-}^*\left(\bigoplus_{\chi \in \scrC_{-}^w} \mcO_W \otimes V(\chi) \right)$ over $X_-$
\end{center}
are pretilting for any $w \in \Z^3$.
Note that these bundles are not generators of $\D(\Qcoh X_{\pm})$ in general.

\subsubsection{The $+$ side}
Recall that the 1-PSs associated to the KN stratification are
\begin{center} 
$\lambda_0 = (1,1,1)$, $\lambda_1 = (1,1,0)$, and $\lambda_2 = (1,0,0)$. 
\end{center}
By definition, a dominant $\GL(S)$-weight $\chi = (\chi_1, \chi_2, \chi_3)$ is contained in $\scrC_+^w$ if and only if
it satisfies
\begin{enumerate}
\item[(a)] $\chi_1 + \chi_2 + \chi_3 \in [w_0, w_0 + 15)$,
\item[(b)] $\chi_{\sigma(1)} + \chi_{\sigma(2)} \in [w_1, w_1 + 8)$ for all $\sigma \in \mathfrak{S}_3$, and
\item[(c)] $\chi_i \in [w_2, w_2+3)$ for all $i = 1,2,3$.
\end{enumerate}
When $w = (w_0,w_1,w_2) = (-7,-4,-1)$, for example, we have
\[ \scrC_+^w = 
\left\{
   \begin{array}{l}
     (-1,-1,-1), (0,-1,-1), (0,0,-1), (0,0,0), (1,-1,-1), \\
     (1,0,-1), (1,0,0), (1,1,-1), (1,1,0), (1,1,1)
   \end{array}
   \right\},
\]
and the bundle
\[ \iota_+^* \left( \bigoplus{}_{\chi \in \scrC_+^w} \,\mcO_W \otimes V(\chi) \right) \in \coh X_+ \]
is a tilting bundle that coincides (up to line bundle twist) with the pull-back of the direct sum of the Kapranov's collection (\S\ref{sect:KapFEC}).
However, the bundle over $X_-$ given by the other restriction
\[ \iota_-^* \left( \bigoplus{}_{\chi \in \scrC_+^w} \,\mcO_W \otimes V(\chi) \right) \in \coh X_- \]
is not (pre)tilting (see Remark~\ref{rem:why not Kapranov}).

\begin{rem}
For any $w \in \Z^3$, the pretilting bundle 
$\scrV_+^w \coloneqq \iota_{+}^*\left(\oplus_{\chi \in \scrC_{+}^w} \mcO_W \otimes V(\chi) \right)$
constructed by Halpern-Leistner's result
cannot coincide with the tilting bundles $\scrT_+^{\star}$ constructed in \S\ref{sec:windows}.
Indeed, by the condition (c) above, 
the bundle $\scrV_+^w$ can contain at most $3$ line bundles,
while the tilting bundles $\scrT_+^{\star}$ contain at least $4$ line bundles.
\end{rem}

\subsubsection{The $-$ side}

Recall that the 1-PSs associated to the KN stratification are
\begin{center} 
$\lambda_0 = (-1,-1,-1)$, $\lambda_1 = (1,1,-4)$, and $\lambda_2 = (1,0,-2)$. 
\end{center}
By definition, a dominant $\GL(S)$-weight $\chi = (\chi_1, \chi_2, \chi_3)$ is contained in $\scrC_-^w$ if and only if
it satisfies
\begin{enumerate}
\item[(a$'$)] $-\chi_1 - \chi_2 - \chi_3 \in [w_0, w_0 + 15)$,
\item[(b$'$)] $\chi_{\sigma(1)} + \chi_{\sigma(2)} - 4 \chi_{\sigma(3)} \in [w_1, w_1 +10)$ for all $\sigma \in \mathfrak{S}_3$, and
\item[(c$'$)] $\chi_{\sigma(1)} - 2 \chi_{\sigma(3)} \in [w_2, w_2+4)$ for all $\sigma \in \mathfrak{S}_3$.
\end{enumerate}
Now it follows from an elementary mathematics that:

\begin{lem}
For $w = (w_0,w_1,w_2) \in \Z^3$,
the set $\scrC_-^w$ of weights is a subset of
\[
\left\{
   \begin{array}{l}
     (-w_2-3,-w_2-3,-w_2-3), (-w_2-2,-w_2-2,-w_2-2), (-w_2-1,-w_2-1,-w_2-1), \\
     (-w_2,-w_2,-w_2), (-w_2-1,-w_2-2,-w_2-2), (-w_2-1,-w_2-1,-w_2-2)
   \end{array}
   \right\}.
\]
In particular, the set $\scrC_-^w$ can contain at most $6$ weights for any $w \in \Z^3$.
\end{lem}

\begin{proof}
Let $\chi$ be a dominant $\GL(3)$-weight, and set $\chi = (x+a+b, x+b, x)$ for some $x \in \Z$ and $a,b \in \Z_{\geq 0}$.
Then the condition (c$'$) implies that, if $\chi \in \scrC_-^w$, then two integers
\begin{center}
$v_{13} \coloneqq (x + a + b) - 2x = -x + (a + b)$ and $v_{31} \coloneqq x - 2(x+a+b) = -x - 2(a+b)$
\end{center}
are both contained in the interval $[w_2,w_2+4)$, which contains exactly four integers.
Therefore the difference $v_{13} - v_{31} = 3(a+b)$ must be at most $3$.
This implies that $(0 \leq)~a + b \leq 1$.

Thus the weight $\chi \in \scrC_-^w$ is one of the following forms:
\begin{enumerate}
\item[(1)] $\chi = (x,x,x)$ for some $x \in \Z$,
\item[(2)] $\chi = (x+1, x,x)$ for some $x \in \Z$, or
\item[(3)] $\chi = (x+1, x+1,x)$ for some $x \in \Z$.
\end{enumerate}

If $\chi = (x,x,x)$ for some $x \in \Z$, then the condition (c$'$) implies that $-x \in [w_2,w_2+4)$.
Therefore $\chi$ should be $(-w_2+k,-w_2+k,-w_2+k)$ for $-3 \leq k \leq 0$.

If $\chi = (x+1, x,x)$ for some $x \in \Z$, then the condition (c$'$) implies that two integers
\begin{center}
$(x+1) -2x = -x+1$ and $x - 2(x+1) = -x - 2$
\end{center}
are both in $[w_2,w_2+4)$.
Thus $-x - 2 = w_2$, and hence $\chi = (-w_2-1,-w_2-2,-w_2-2)$.

If $\chi = (x+1, x+1,x)$ for some $x \in \Z$, then the same argument shows that
the weight $\chi$ must be $(-w_2-1,-w_2-1,-w_2-2)$.
\end{proof}

As a consequence, the pretilting bundle
$\scrV_-^w \coloneqq \iota_{-}^*\left(\bigoplus_{\chi \in \scrC_{-}^w} \mcO_W \otimes V(\chi) \right)$
cannot be tilting, since it cannot have $10$ direct summands, where $10$ is the rank of the Grothendieck group
$K_0(X_-)$.
In other words, the category $\scrG_-^w$ is not ``magic'' for all $w \in \Z^3$.

If $w = (-7, -5, -2)$, then the equality is attained:
\[ \scrC_{-}^w = \{ (-1,-1,-1), (0,0,0),(1,0,0),(1,1,0),(1,1,1),(2,2,2) \}. \]
In this case, 
\[ \scrV_-^w \simeq \mcO(1) \oplus \mcO \oplus \scrP \oplus \scrP^{\vee}(-1) \oplus \mcO(-1) \oplus \mcO(-2). \]
The fact that $\scrV_-^w$ is pretilting can recover some computations in \S\ref{sect:minus-side-tilting},
but not all of them.

\appendix
\section{On fully faithful functors} \label{sect:fffunctors}

This section collects two propositions, which will be well-known to experts.
We provide a full proof of those propositions, since we could not find references.

\begin{prop} \label{prop:generator ff}
Let $\Phi \colon \scrA \to \scrB$ be an exact functor between triangulated categories.
Assume that $\scrA$ admits a classical generator $G \in \scrA$ such that the homomorphism
\[ \Hom(G,G[i]) \to \Hom(\Phi G, \Phi G[i]) \]
associated to $\Phi$ is an isomorphism for all $i \in \Z$.
Then $\Phi$ is fully faithful.
\end{prop}

\begin{proof}
Define subcategories $\langle G \rangle_n \subset \scrC$ for $n \geq 0$ as in \cite[\S2]{BFK11}.
Since $G$ is a classical generator, $\scrA = \thick(G) = \cup_{i \geq 0} \langle G \rangle_n$.
We claim that, for any non-negative integer $n \geq 0$,
the natural homomorphism
\[ \Hom(x',x) \to \Hom(\Phi x', \Phi x) \]
is an isomorphism for all $x, x' \in \langle G \rangle_n$, and prove this claim using an induction on $n$.

The case $n = 0$ follows from the assumption.

Let us assume that $n > 0$, and that the claim holds for $n -1$.
Fix an arbitrary object $x \in \langle G \rangle_n$.

Then there exist $x_1 \in \langle G \rangle_n$, $y \in \langle G \rangle_{n-1}$, $z \in \langle G \rangle_{0}$,
and an exact triangle
\[ y \to x \oplus x_1 \to z \to y[1]. \]
Let $y' \in \langle G \rangle_{n-1}$ be an arbitrary object,
and consider the following commutative diagram of long exact sequences.
\[ \footnotesize{\begin{tikzcd}
\Hom(y', z[-1]) \arrow[r] \arrow[d, "\simeq"] & \Hom(y', y) \arrow[r] \arrow[d, "\simeq"] & \Hom(y' ,x\oplus x_1) \arrow[r] \arrow[d]  & \Hom(y', z) \arrow[r] \arrow[d, "\simeq"] & \Hom(y', y[1])  \arrow[d, "\simeq"] \\
\Hom(\Phi y', \Phi z[-1]) \arrow[r] & \Hom(\Phi y', \Phi y) \arrow[r] & \Hom(\Phi y' ,\Phi x\oplus \Phi x_1) \arrow[r]  & \Hom(\Phi y', \Phi z) \arrow[r] & \Hom(\Phi y', \Phi y[1])
\end{tikzcd}} \]
By the inductive hypothesis, all vertical morphisms except for the middle one are isomorphisms.
Therefore the five lemma implies that the middle one is also an isomorphism.
As a consequence, it has followed that for any $y' \in \langle G \rangle_{n-1}$,
the morphism $\Hom(y',x) \to \Hom(\Phi y', \Phi x)$ is isomorphic.

Similarly, let $x' \in \langle G \rangle_n$ be an arbitrary object, and assume that there exists an exact triangle
\[ w \to x' \oplus x_2 \to u \to w[1] \]
for some $x_2 \in \langle G \rangle_n$, $w \in \langle G \rangle_{n-1}$ and $u \in \langle G \rangle_{0}$.
Applying the functor $\Hom(-,x)$ to the sequence above, and then using the same argument as above,
it follows that the morphism $\Hom(x',x) \to \Hom(\Phi x', \Phi x)$ is an isomorphism.

Since $x$ and $x'$ are arbitrary, the claim for $n$ follows.
\end{proof}

\begin{prop} \label{prop: dense im}
Let $\Phi \colon \scrA \to \scrB$ be an exact fully faithful functor between triangulated categories.
Assume that the category $\scrA$ is idempotent-complete.
Then the essential image of $\Phi$ is closed under taking direct summands.
(In other words, the essential image is a thick subcategory of $\scrB$.)
\end{prop}

\begin{proof}
Assume that $\Phi x \simeq y \oplus z$ for $x \in \scrA$ and $y,z \in \scrB$.
Let $e_y \colon \Phi x \to \Phi x$ be an idempotent associated to the summand $y$.
Since $\Phi$ is fully faithful,
there is a unique morphism $e \colon x \to x$ such that $\Phi(e) = e_y$, and this $e$ remains an idempotent.
Then the idempotent-completeness of $\scrA$ implies that there are $w, u \in \scrA$ such that $x \simeq w \oplus u$
and the idempotent $e$ coincides with the composite $x \twoheadrightarrow w \hookrightarrow x$.
Now it is standard that $\Phi w \simeq y$.
Therefore the essential image of $\Phi$ is closed under taking direct summands.
\end{proof}

\section{Code}\label{app:code}
In this section we outline the computational tools employed for the Borel-Weil-Bott computations presented in this paper. The goal here is primarily explanatory: to keep the exposition uncluttered, we omit certain auxiliary functions that do not affect the core ideas. The complete and fully documented code is available in the public repository \cite{repository}.
\subsection{Requirements}
\begin{enumerate}
    \item Python $3.12$ must be installed \cite{python}.
    \item Run the following command lines in the terminal.
    \begin{lstlisting}[style=terminal]
        $ pip install lrcalc
    \end{lstlisting}
    This will install the package \verb|lrcalc| which we use to compute the Littlewood--Richardson product of partitions \cite{lrcalc}.
\end{enumerate}
\subsection{Main classes}
In the following, we describe the classes implementing homogeneous vector bundles on Grassmannians. We introduce methods to compute their products, duals, ranks and cohomology.
\begin{lstlisting}[style=modernpython]
    import lrcalc
    
    class Partition:
        """
        We made a custom class so that we can use '*' to denote the Littlewood--Richardson product
        (by the special method __mul__), to keep track of multiplicities and to have a clean print.
        Multiplicity is one by default
        """
        def __init__(self, mylist, multiplicity=1):
            self.partition = mylist
            self.multiplicity= multiplicity
        def __mul__(self, other_partition):
            raw_result = lrcalc.mult(self.partition, other_partition.partition)
            base_multiplicity = self.multiplicity * other_partition.multiplicity
            result = []
            for key, value in raw_result.items():
                result.append(Partition(list(key), value * base_multiplicity))
            return result
        def __str__(self):
            return f"{self.partition}(x{self.multiplicity})"
        def __repr__(self):
            return f"{self.partition}(x{self.multiplicity})"
    
    class HomogeneousIrreducible:
    """
    this class represents homogeneous irrediucible vector bundles on G(k, n).
    the first partition is a list representing a Schur power of U*, while
    the second represents a Schur power of Q
    Multiplicity is one by default
    """
    def __init__(self, k, n, first_partition, second_partition, multiplicity=1):
        self.k = k
        self.n = n
        """
        beware that self.first_partition is a Partition object. To access the partition as a list,
        call self.first_partition.partition. Same for second_partition
        """
        self.first_partition = Partition(first_partition + [0] * (k - len(first_partition)), 1)
        self.second_partition = Partition(second_partition + [0] * (n - k - len(second_partition)), 1)
        self.multiplicity= multiplicity
    def __mul__(self, other_double_partition):
        """
        we compute the product of two HomogeneousVectorBundle's with '*'. Littlewood--Richardson is applied
        on both partitions.
        """
        result = []
        result_first_partitions = self.first_partition * other_double_partition.first_partition
        result_second_partitions = self.second_partition * other_double_partition.second_partition
        for item in result_first_partitions:
            for other_item in result_second_partitions:
                if len(item.partition) <= self.k and len(other_item.partition) <= self.n-self.k:
                    total_multiplicity = self.multiplicity * other_double_partition.multiplicity * item.multiplicity * other_item.multiplicity
                    result.append(HomogeneousIrreducible(self.k, self.n, item.partition, other_item.partition, total_multiplicity))
        return HomogeneousDirectSum(result)
    def __str__(self):
        return (
            f"{self.first_partition.partition}|{self.second_partition.partition}"
            f"(x{self.multiplicity})"
    )
    def __repr__(self):
        return (
            f"{self.first_partition.partition}|{self.second_partition.partition}"
            f"(x{self.multiplicity})"
    )
    def rank(self):
        return utils.weyl_dim(self.first_partition.partition) * utils.weyl_dim(self.second_partition.partition)
    def dual(self):
        # reverse and negate
        part1 = self.first_partition.partition
        part2 = self.second_partition.partition
        dual_first = [-x for x in reversed(part1)]
        dual_second = [-x for x in reversed(part2)]
        # find the global twist
        max_entry = -min(dual_first + dual_second)
        t = max_entry
        # shift by k
        twisted_first = [x + t for x in dual_first]
        twisted_second = [x + t for x in dual_second]
        # return a new HomogeneousIrreducible (with same multiplicity)
        return HomogeneousIrreducible(self.k, self.n, twisted_first, twisted_second, self.multiplicity)

    def cohomology(self):
        """
        compute cohomology with Borel--Weyl--Bott.
        If rep is true, the output has the form {"representation": [,,,], "dimension": k, "degree": p},
        otherwise  {"representation": [,,,], "degree": p}
        """
        part1 = self.first_partition.partition
        part2 = self.second_partition.partition
        padded_first = part1 + [0] * (self.k - len(part1))
        padded_second = part2 + [0] * (self.n - self.k - len(part2))
        total_partition = padded_first + padded_second
        # adding rho
        total_partition = utils.add_rho(total_partition)
        # checking repetitions.
        # if there are none:
        if not utils.has_repeats(total_partition):
            degree = 0
            while True:
                if utils.is_decreasing(total_partition):
                    total_partition = utils.subtract_rho(total_partition)
                    return {
                        "representaiton": utils.normalize_partition(total_partition),
                        "dimension": utils.weyl_dim(total_partition),
                        "degree": degree
                    }
                else:
                    utils.swap_first_increase(total_partition)
                    degree = degree + 1
        # acyclic case
        else:
            return "no cohomology"

    
    
    class HomogeneousDirectSum(list):
        """
        this class is just a wrapper for HomogeneousIrreducible. We made it so that
        the tensor product is distributive wrt direct sum
        """
        def __init__(self, summands=None):
            """
            summands: an iterable of HomogeneousIrreducible (or other DirectSums).
            """
            # if they passed nothing, treat as empty list
            super().__init__(summands or [])
        def __mul__(self, other):
            """
            Distribute self * other over every summand.
            other may be a single HI or another direct sum.
            """
            out = []
            # if other is a 'vector':
            if isinstance(other, HomogeneousDirectSum):
                for A in self:
                    for B in other:
                        # A*B returns another direct sum
                        prod = A * B
                        out.extend(prod if isinstance(prod, list) else [prod])
            else:
                # other is a single HI
                for A in self:
                    prod = A * other
                    out.extend(prod if isinstance(prod, list) else [prod])
            return HomogeneousDirectSum(out)
    
        def dual(self):
            """
            Return the direct sum of the duals of each summand.
            """
            return HomogeneousDirectSum([X.dual() for X in self])
    
        def __rmul__(self, other):
            """
            Called when the left operand doesn't know how to multiply
            by a direct sum.  Typically other is a single HI.
            """
            # other * self  = distribute other over our list
            out = []
            for A in self:
                prod = other * A
                out.extend(prod if isinstance(prod, list) else [prod])
            return HomogeneousDirectSum(out)
    
        def cohomology(self):
            """
            cohomology of a direct sum is the direct sum of the cohomologies
            """
            return [X.cohomology() for X in self]

    
    
\end{lstlisting}

\subsection{The script for \S\ref{sect:plus-side-tilting}}

\subsubsection{The script for the proof of Theorem~\ref{lem:tilting_plus_side}} \label{code:lem:tilting_plus_side}
    The following function computes the Ext between two pullbacks of direct sums of two homogeneous vector bundles on either $X_+$ or~$X_-$. The input data is organized as follows:
    \begin{itemize}
        \item The integers \verb|k| and \verb|n| are the ones defining the Grassmannian $G(k, n)$.
        \item \verb|E| and \verb|F| are instances of \verb|HomogeneousIrreducible| or \verb|HomogeneousDirectSum|.
        \item The variable \verb|bundle| is expected to be either \verb|''U*(-2)''| or \verb|''Q*(-2)''|, depending on whether we want to compute the Ext on $X_+$ or $X_-$.
        \item The variable \verb|cutoff| is a nonnegative integer representing the maximum symmetric power of $\mcU(2)$ or $\mcQ^\vee(2)$ we consider in the expansion of the pushforward of the structure sheaf of the total space (see the proof of Theorem \ref{lem:tilting_plus_side}).
    \end{itemize}
    \begin{lstlisting}[style=modernpython]
        def ext_on_total_space(k, n, E, F, bundle, cutoff=8):
            pushforward_structure_sheaf = []
            if bundle == 'U*(-2)':
                for i in range(cutoff + 1):
                    pushforward_structure_sheaf.append(hb.HomogeneousIrreducible(k, n, [2*i]*(k-1)+[i], [0]*(n-k), 1))
                pushforward_structure_sheaf = hb.HomogeneousDirectSum(pushforward_structure_sheaf)
                # return pushforward_structure_sheaf
            elif bundle == 'Q(-2)':
                for i in range(cutoff + 1):
                    pushforward_structure_sheaf.append(hb.HomogeneousIrreducible(k, n, [2*i]*(k), [i]+[0]*(n-k-1), 1))
                pushforward_structure_sheaf = hb.HomogeneousDirectSum(pushforward_structure_sheaf)
                # return pushforward_structure_sheaf
            else:
                return "ERROR: this bundle is not supported (yet)"
            return (E.dual() * F * pushforward_structure_sheaf).cohomology()
    \end{lstlisting}
    To check whether a given bundle is tilting, we can simply run the following function:
    \begin{lstlisting}[style=modernpython]
        def is_it_tilting(k, n, E, bundle, cutoff):
            check = True
            cohomology = ext_on_total_space(k, n, E, E, bundle, cutoff)
            for item in cohomology:
                if item != "no cohomology":
                    if item["degree"] != 0:
                        check = False
            return check
    \end{lstlisting}
    The following code blocks introduce the generators of the windows as \verb|HomogeneousDirectSum| objects: 
    \begin{lstlisting}[style=modernpython]
        FirstWindowG35 = hb.HomogeneousDirectSum([
            hb.HomogeneousIrreducible(3, 5, [], [], 1),
            hb.HomogeneousIrreducible(3, 5, [1,1,1], [], 1),
            hb.HomogeneousIrreducible(3, 5, [], [1,1], 1),
            hb.HomogeneousIrreducible(3, 5, [1,1], [], 1),
            hb.HomogeneousIrreducible(3, 5, [1], [1,1], 1),
            hb.HomogeneousIrreducible(3, 5, [2], [1,1], 1),
            hb.HomogeneousIrreducible(3, 5, [1,1,1], [], 1),
            hb.HomogeneousIrreducible(3, 5, [2,1,1], [], 1),
            hb.HomogeneousIrreducible(3, 5, [2,2,1], [], 1),
            hb.HomogeneousIrreducible(3, 5, [2,2,2], [], 1),
            ]
        )
    \end{lstlisting}
    \begin{lstlisting}[style=modernpython]
        SecondWindowG35 = hb.HomogeneousDirectSum([
            hb.HomogeneousIrreducible(3, 5, [], [1,1], 1),
            hb.HomogeneousIrreducible(3, 5, [], [], 1),
            hb.HomogeneousIrreducible(3, 5, [1,1,1], [], 1),
            hb.HomogeneousIrreducible(3, 5, [2,2,2], [], 1),
            hb.HomogeneousIrreducible(3, 5, [3,3,3], [], 1),
            hb.HomogeneousIrreducible(3, 5, [1], [1,1,1], 1),
            hb.HomogeneousIrreducible(3, 5, [1], [], 1),
            hb.HomogeneousIrreducible(3, 5, [2,1,1], [], 1),
            hb.HomogeneousIrreducible(3, 5, [1,1], [], 1),
            hb.HomogeneousIrreducible(3, 5, [2,2,1], [1,1], 1),
            ]
        )
    \end{lstlisting}
    Then, the proof is completed once we run the following:
    \begin{lstlisting}[style=modernpython]
        print(f"Is the first generator tilting? {is_it_tilting(3, 5, FirstWindowG35, "U*(-2)")}.")
        print(f"Is the second generator tilting? {is_it_tilting(3, 5, SecondWindowG35, "U*(-2)")}.")

    \end{lstlisting}

\subsection{The scripts for \S\ref{sect:minus-side-tilting}}

We use an adaptation of the function \verb|ext_on_total_space| we defined in \S\ref{sec:windows}. In particular, we rely on the following script:
\begin{lstlisting}[style=modernpython]
    def higher_ext_on_total_space(k, n, E, F, bundle, cutoff=8):
    check = True
    cohomology = ext_on_total_space(k, n, E, F, bundle, cutoff)
    for item in cohomology:
        if item != "no cohomology":
            if item["degree"] != 0:
                check = False
    return check 

\end{lstlisting}

\subsubsection{The script for the proof of Proposition~\ref{prop:easy_vanishings_line_bundles}}
\label{code:prop:easy_vanishings_line_bundles}

    \begin{lstlisting}[style=modernpython]
        for k in range(5):
            line_bundle = hb.HomogeneousIrreducible(2, 5, [0], [k,k,k])
            print(f"O({-k}) has no higher cohomology: {ext.higher_cohomlogy_on_total_space(2, 5, line_bundle, 'Q(-2)')}")
    \end{lstlisting}
    
\subsubsection{The script for the proof of Proposition~\ref{P}}
\label{code:P}
    \begin{lstlisting}[style=modernpython]
    for k in range(3):
        vector_bundle = hb.HomogeneousIrreducible(2, 5, [1], [k,k,k])
        print(f"U^*({-k}) has no higher cohomology: {ext.higher_cohomlogy_on_total_space(2, 5, vector_bundle, 'Q(-2)')}")
\end{lstlisting}

\subsubsection{The script for the proof of Proposition~\ref{prop P sym2P}}
\label{code:prop P sym2P}

\begin{lstlisting}[style=modernpython]
    for a in range(3):
        twisted_sym = hb.HomogeneousIrreducible(2, 5, [2+a,a], [])
        universal = hb.HomogeneousIrreducible(2, 5, [1], [])
        print(f"vanishing of higher Ext(U_2^*, Sym^2 U_2^*({a})) on G(2, 5): {ext.higher_ext_on_total_space(2, 5, universal, twisted_sym, 'Q(-2)')}")   
\end{lstlisting}

\subsubsection{The script for the proof of Proposition~\ref{prop:sym2P-sym2P}}
\label{code:prop:sym2P-sym2P}

\begin{lstlisting}[style=modernpython]
    sym = hb.HomogeneousIrreducible(2, 5, [2], [])
    dual_sym = hb.HomogeneousIrreducible(2, 5, [2], [2,2,2])
    print(f"Sym^2 U_2\otimes \Sym^2 U_2^* has no higher cohomology: {ext.higher_cohomlogy_on_total_space(3, 5, vector_bundle, 'U*(-2)')}")
\end{lstlisting}

\bibliographystyle{alpha}
\bibliography{biblio}

\end{document}